\documentclass[a4paper,10pt, oneside]{article}
\usepackage{amsmath, amssymb, amsthm,graphicx}
\usepackage[font=small,labelfont=md,textfont=it]{caption}
\usepackage{floatrow,float}
\usepackage[titletoc, title]{appendix}
\usepackage[colorlinks,linkcolor=blue,citecolor=blue]{hyperref}
\usepackage{longtable}
\usepackage{pgfplots}
\usepackage{diagbox}
\usepackage{booktabs,makecell,multirow}
\usepackage[capitalise,sort,compress]{cleveref}
\usepackage{cases,color}
\crefname{equation}{}{}
\crefname{lem}{Lemma}{Lemmas}
\crefname{thm}{Theorem}{Theorems}

\DeclareMathOperator{\D}{D}

\newcommand{\dual}[1]{\langle {#1} \rangle}
\newcommand{\Dual}[1]{\left\langle {#1} \right\rangle}
\newcommand{\jmp}[1]{{[\![ {#1} ]\!]}}
\newcommand{\nm}[1]{\lVert {#1} \rVert}

\newcommand{\snm}[1]{\lvert {#1} \rvert}
\newcommand{\Snm}[1]{\left\lvert {#1} \right\rvert}
\newcommand{\ssnm}[1]
{
  \left\vert\kern-0.25ex
  \left\vert\kern-0.25ex
  \left\vert
  {#1}
  \right\vert\kern-0.25ex
  \right\vert\kern-0.25ex
  \right\vert
}

\makeatletter
\def\spher@harm#1{%
  \vbox{\hbox{%
    \offinterlineskip
    \valign{&\hb@xt@2\p@{\hss$##$\hss}\vskip.2ex\cr#1\crcr}%
  }\vskip-.36ex}%
}
\def\gshone{\spher@harm{.}}
\def\gshtwo{\spher@harm{.&.}}
\def\gshthree{\spher@harm{.&.&.}}
\let\gsh\spher@harm
\makeatother

\newtheorem{Def}{Definition}[section]

\newtheorem{lem}{Lemma}[section]
\newtheorem{rem}{Remark}[section]
\newtheorem{thm}{Theorem}[section]

\makeatletter\def\@captype{table}\makeatother

\begin{document}

\title{
	\Large \bf Analysis of a time-stepping discontinuous Galerkin
	method for fractional diffusion-wave equation
	with nonsmooth data
	\thanks
	{
		This work was supported in part
		by National Natural Science Foundation
		of China (11771312).
	}
}
\author{
	Binjie Li \thanks{Email: libinjie@scu.edu.cn},
	Tao Wang \thanks{Corresponding author. Email: wangtao5233@hotmail.com},	Xiaoping Xie \thanks{Email: xpxie@scu.edu.cn} \\
	{School of Mathematics, Sichuan University, Chengdu 610064, China}
}

\date{}
\maketitle

\begin{abstract}
	This paper analyzes a time-stepping discontinuous Galerkin method for fractional
	diffusion-wave problems. This method uses piecewise constant functions in the
	temporal discretization and continuous piecewise linear functions in the spatial
	discretization. Nearly optimal convergence rate with respect to the regularity of
	the solution is established when the source term is nonsmooth, and nearly optimal
	convergence rate $ \ln(1/\tau)(\sqrt{\ln(1/h)}h^2+\tau) $ is derived under
	appropriate regularity assumption on the source term. Convergence is also
	established without smoothness assumption on the initial value. Finally, numerical
	experiments are performed to verify the theoretical results.
\end{abstract}

\medskip\noindent{\bf Keywords:}
fractional diffusion-wave problem,
discontinuous Galerkin method, discrete Laplace transform,
convergence,
nonsmooth data.

\section{Introduction}
This paper considers the following time fractional diffusion-wave problem:
\begin{equation}
	\label{eq:model}
	\left\{
		\begin{aligned}
			u' - \Delta \D_{0+}^{-\alpha} u & = f   &  & \text{in $~~ \Omega\times(0,T)$,}          \\
			u                               & = 0   &  & \text{on $ \partial\Omega \times (0,T) $,} \\
			u(0)                            & = u_0 &  & \text{in $~~ \Omega $,}
		\end{aligned}
	\right.
\end{equation}
where $ 0<\alpha<1 $, $ 0 < T < \infty $, $ \Omega \subset \mathbb R^d $ ($d=1,2,3$)
is a convex $ d $-polytope, $ \D_{0+}^{-\alpha} $ is a Riemann-Liouville fractional
integral operator of order $ \alpha $, and $ f $ and $ u_0 $ are two given functions.
The above fractional diffusion-wave equation also belongs to the class of evolution
equations with a positive-type memory term (or integro-differential equations with a
weakly singular convolution kernel), which have attracted many works in the past
thirty years.

%The above fractional diffusion-wave equation is also called an evolution equation
%with a positive-type memory term or an integro-differential equation with a weakly
%singular convolution kernel.

%The above fractional diffusion-wave equation is a particular case of evolution
%equations with a positive-type memory term, and it is also an integro-differential
%equation with a weakly singular convolution kernel.

Let us first briefly summarize some works devoted to the numerical treatments of
problem \cref{eq:model}. McLean and Thom\'ee~\cite{McLean1993} proposed and analyzed
two discretizations: the first uses the backward Euler method to approximate the
first-order time derivative and a first-order integration rule to approximate the
fractional integral; the second uses a second-order backward difference scheme to
approximate the first-order time derivative and a second-order integration rule to
approximate the fractional integral. Then McLean et al.~\cite{McLean1996} analyzed
two discretizations with variable time steps: the first is a simple variant of the
first one analyzed in \cite{McLean1993}; the second combined the Crank-Nicolson
scheme and two integral rules to approximate the fractional integral (but the
temporal accuracy is not better than $ \mathcal O(\tau^{1+\alpha}) $). Combining the
first-order and second-order backward difference schemes and the convolution
quadrature rules \cite{Lubich1986}, Lubich et al.\cite{Lubich1996} proposed and
analyzed two discretizations for problem \cref{eq:model}, where optimal order error
bounds were derived for positive times without spatial regularity assumption on the
data. Cuesta et al.~\cite{Cuesta2006} proposed and studied a second-order
discretization for problem \cref{eq:model} and its semilinear version.

Representing the solution as a contour integral by the Laplace transform techinque
and approximating this contour integral, McLean and Thom\'ee
\cite{McLean2010-B,McLean2010} developed and analyzed three numerical methods for
problem \cref{eq:model}. These methods use $ 2N+1 $ quadrature points, and the first
method possesses temporal accuracies $ \mathcal O(e^{-cN}) $ away from $ t=0 $, the
second and third have temporal accuracy $ \mathcal O(e^{-c\sqrt N}) $.

%For their third
%method, they derived the following error estimates:
%\begin{small}
%\begin{align*}
%& \nm{U_{N,h}(t) - u_h(t)}_{L^2(\Omega)} \leqslant
%C e^{-C_0 \sqrt N} \Big(
%\nm{f(0)}_{\dot H^{2\sigma_0}(\Omega)} +
%\int_0^t \nm{f'(s)}_{\dot H^{2\sigma_0}(\Omega)} \, \mathrm{d}s
%\Big) \\
%& \qquad {} +
%Ch^2 \Big(
%t^{(1+\alpha)\sigma-\alpha} \nm{f(0)}_{\dot H^{2\sigma}(\Omega)} +
%\sum_{j=1}^3 \int_0^t s^{j-\alpha} \nm{f^{(j)}(s)}_{L^2(\Omega)} \, \mathrm{d}s
%\Big), \\
%& \nm{u_h(t) - U_{N,h}(t)}_{L^\infty(\Omega)}  \leqslant
%C e^{-C_0\sqrt{N}} \Big(
%\nm{f(0)}_{L^\infty(\Omega)} +
%\int_0^t \nm{f'(s)}_{L^\infty(\Omega)} \, \mathrm{d}s
%\Big) \\
%& \quad {} +
%Ch^2 \ln(1/h)^2
%\int_0^t \nm{\Delta u'(s)}_{L^\infty(\Omega)} \, \mathrm{d}s,
%\end{align*}
%\end{small}
%where $ \sigma $ and $ \sigma_0 $ are two positive constants satisfying some
%inequalities.

McLean and Mustapha \cite{McLean2007} studied a generalized Crank-Nicolson scheme
for problem \cref{eq:model}, and they obtained accuracy order $ \mathcal O(h^2 +
\tau^2) $ on appropriately graded temporal grids under the condition that the
solution and the forcing term satisfy some growth estimates. Mustapha and McLean
\cite{Mustapha2009Discontinuous} applied the famous time-stepping discontinuous
Galerkin (DG) method \cite[Chapter 12]{Thomee2006} to an evolution equation with a
memory term of positive type. For the low-order DG method, they derived the accuracy
order $ \mathcal O(\ln(1/\tau)h^2 + \tau) $ on appropriately graded temporal grids
under the condition that the time derivatives of the solution satisfy some growth
estimates. We notice that this low-order DG method is identical to the first-order
discretization analyzed in the aforementioned work \cite{McLean1996}. They also
analyzed an $hp$-version of the DG method in \cite{mustapha2014well-posedness}. So
far, by our knowledge the convergence of this algorithm has not been established with
nonsmooth data.

%Cuesta and Palencia \cite{Cuesta2003} proposed a temporal semi-discretization based
%on the trapezoidal rule and derived second order convergence under the condition that
%the solution is sufficiently smooth.

%For more related works, we refer the reader to
%\cite{Cheng1992,Adolfsson2003,Cuesta2003}.

%The analysis in \cite{McLean1993} uses a combination of energy arguments and Fourier
%analysis in the time variable, and requires constant time steps. In \cite{McLean1996}
%extensions to variable time steps were studied using only energy methods. The results
%show optimal order error bounds which hold uniformly for small $ t $, under specified
%smoothness assumptions.

This paper analyzed the convergence of the aforementioned low-order DG method, which
is a further development of the works in \cite{McLean1996,Mustapha2009Discontinuous}.
For $ f = 0 $, we derive the error estimate
\[
	\nm{u(t_j) - U_j}_{L^2(\Omega)} \leqslant
	C(h^2 t_j^{-\alpha-1} + \tau t_j^{-1}) \nm{u_0}_{L^2(\Omega)}.
\]
For $ u_0 = 0 $, we obtain the following error estimates:
\begin{align*}
	\nm{u-U}_{L^\infty(0,T;L^2(\Omega))} & \leqslant
	C \big(h + \sqrt{\ln(1/h)}\,\tau^{1/2}\big)
	\nm{f}_{L^2(0,T;\dot H^{\alpha/(\alpha+1)}(\Omega))}, \\
	\nm{u-U}_{L^\infty(0,T;L^2(\Omega))} & \leqslant
	C \ln(T/\tau) \big(\sqrt{\ln(1/h)}h^2 + \tau\big)
	\nm{f}_{{}_0H^{\alpha+1/2}(0,T;L^2(\Omega))},
\end{align*}
where the first is nearly optimal with respect to the regularity of the solution and
the second is nearly optimal, and we notice that since $ \alpha/(\alpha+1) < 1/2 $
the first estimate imposes no boundary condition on $ f $. In addition, to
investigate the effect of the nonvanishing $ f(0) $ on the accuracy of the numerical
solution, we establish the error estimate
\[
	\nm{u(t_j) - U_j}_{L^2(\Omega)} \leqslant
	C (t_j^{-\alpha} h^2 + \tau) \nm{v}_{L^2(\Omega)},
\]
in the case that $ u_0 = 0 $ and $ f(t) = v \in L^2(\Omega) $, $ 0 \leqslant t
\leqslant T $.

%a discrete
%regularity estimate by the Laplace transform technique. Then we derive the following
%error estimates:
%where the first error estimate is nearly optimal with respect to the regularity of
%the solution, and this estimate imposes no boundary condition on $ f $ since $
%\alpha/(\alpha+1) < 1/2 $.

%this paper answers the
%following fundamental questions: how to derive optimal convergence rate with respect
%to the regularity of the solution with nonsmooth source term; how to establish sharp
%error estimate in the case the initial value is nonsmooth; with what regularity
%restrictions on the source term can optimal convergence rate be derived; how do the
%nonvanishing $ f(0) $ affect the accuracy of the numerical solution.

%Using the Laplace transform technique, McLean andd Mustapha \cite{McLean2015Time}
%analyzed the convergence of a temporal semi-discretization for time fractional
%diffusion problems with vanishing source term. By the Laplace transform technique,
%we establish an important discrete regularity estimate.

%Mustapha and Sch\"otzau \cite{mustapha2014well-posedness} proposed an $hp$-version
%discontinuous Galerkin method for problem \cref{eq:model}, and, under appropriate
%regularity assumptions on the solution, they obtained exponential rates of
%convergence in time direction on suitable graded temporal grids. So far, by our
%knowledge the convergence of this algorithm has not been established with nonsmooth
%data.

The rest of this paper is organized as follows. \cref{sec:pre} introduces some
Sobolev spaces, the fractional calculus operator, a time-stepping discontinuous
Galerkin method, the weak solution of problem \cref{eq:model} and its regularity.
\cref{sec:disc_regu} investigates two discretizations of two fractional ordinary
equations, respectively. \cref{sec:main} establishes the convergence of the
numerical method. \cref{sec:numer} performs four numerical experiments to confirm the
theoretical results. Finally, \cref{sec:conclusion} provides some concluding remarks.

\section{Preliminaries}
\label{sec:pre}
\subsection{Sobolev spaces}
Assume that $ -\infty < a < b < \infty $. For each $ m \in \mathbb N $, define
\begin{align*}
	{}_0H^m(a,b) & := \{v\in H^m(a,b): v^{(k)}(a)=0,\,\, 0\leqslant k<m\}, \\
	{}^0H^m(a,b) & := \{v\in H^m(a,b): v^{(k)}(b)=0,\,\, 0\leqslant k<m\},
\end{align*}
where $ H^m(a,b) $ is a usual Sobolev space \cite{Tartar2007} and $ v^{(k)} $ is the
$ k $-th weak derivative of $ v $. We equip the above two spaces with the norms
\begin{align*}
	\nm{v}_{{}^0H^m(a,b)} &:= \nm{v^{(m)}}_{L^2(a,b)}
	\quad \forall v \in {}^0H^m(a,b), \\
	\nm{v}_{{}_0H^m(a,b)} &:= \nm{v^{(m)}}_{L^2(a,b)}
	\quad \forall v \in {}_0H^m(a,b),
\end{align*}
respectively. For any $ m \in \mathbb N_{>0} $ and $ 0 < \theta < 1 $, define
\begin{align*}
	{}_0H^{m-\theta}(a,b) & := [ {}_0H^{m-1}(a,b),\ {}_0H^m(a,b) ]_{1-\theta,2}, \\
	{}^0H^{m-\theta}(a,b) & := [ {}^0H^{m-1}(a,b),\ {}^0H^m(a,b) ]_{1-\theta,2},
\end{align*}
where $ [\cdot, \cdot]_{\theta,2} $ means the famous $ K $-method \cite[Chapter
22]{Tartar2007}. For $ 0 < \gamma < \infty $, we use $ {}^0H^{-\gamma}(a,b) $ and $
{}_0H^{-\gamma}(a,b) $ to denote the dual spaces of $ {}_0H^\gamma(a,b) $ and $
{}^0H^\gamma(a,b) $, respectively. Conversely, since $ {}_0H^\gamma(a,b) $ and $
{}^0H^\gamma(a,b) $ are reflexive, they are the dual spaces of $ {}^0H^{-\gamma}(a,b)
$ and $ {}_0H^{-\gamma}(a,b) $, respectively. Moreover, for any $ 0 < \gamma < 1/2 $,
$ {}_0H^\gamma(a,b) = {}^0H^\gamma(a,b) = H^\gamma(a,b) $ with equivalent norms
(cf.~\cite[Chapter 1]{Lions1972}), and hence $ {}_0H^{-\gamma}(a,b) =
{}^0H^{-\gamma}(a,b) $ with equivalent norms.

It is well known that there exists an orthonormal basis $\{\phi_n: n \in \mathbb N
\}$ of $ L^2(\Omega) $ such that
\[
	\left\{
		\begin{aligned}
			-\Delta \phi_n ={} &\lambda_n \phi_n&&\,
			{\rm~in~}~\,\,\Omega,\\
			\phi_n={}&0&&{\rm~on~}\partial\Omega,
		\end{aligned}
	\right.
\]
where $ \{ \lambda_n: n \in \mathbb N \} $ is a positive non-decreasing sequence and
$\lambda_n\to\infty$ as $n\to\infty$. For any $ -\infty< \beta < \infty $, define
\begin{center}
	$
	\dot H^\beta(\Omega) := \Big\{
		\sum_{n=0}^\infty v_n \phi_n:\
		\sum_{n=0}^\infty \lambda_n^\beta v_n^2 < \infty
	\Big\}
	$,
\end{center}
and endow this space with the norm
\begin{center}
	$
	\big\|\sum_{n=0}^\infty v_n \phi_n  \big\|_{\dot H^\beta(\Omega)}
	:= \Big(
		\sum_{n=0}^\infty \lambda_n^\beta v_n^2
	\Big)^{1/2}
	$.
\end{center}
For any $ \beta,\gamma \in \mathbb R $, define
\[
	{}^0H^\gamma(a,b;\dot H^\beta(\Omega)) := \bigg\{
		\sum_{n=0}^\infty c_n \phi_n:\
		\sum_{n=0}^\infty \lambda_n^\beta \nm{c_n}_{{}^0H^\gamma(a,b)}^2 < \infty
	\bigg\},
\]
and equip this space with the norm
\[
	\Big\| \sum_{n=0}^\infty c_n \phi_n \Big\|_{
		{}^0H^\gamma(a,b;\dot H^\beta(\Omega))
	} :=
	\bigg(
		\sum_{n = 0}^\infty \lambda_n^\beta
		\nm{c_n}_{{}^0H^\gamma(a,b)}^2
	\bigg)^{1/2}.
\]
The space $ {}_0H^\gamma(a,b;\dot H^\beta(\Omega)) $ is analogously defined, and it
is evident that $ {}^0H^{-\gamma}(a,b;\dot H^{-\beta}(\Omega)) $ is the dual space of
$ {}_0H^\gamma(a,b;\dot H^\beta(\Omega)) $ in the sense that
\[
	\Dual{
		\sum_{n=0}^\infty c_n \phi_n, \sum_{n=0}^\infty d_n \phi_n
	}_{{}_0H^\gamma(a,b;\dot H^\beta(\Omega))} :=
	\sum_{n=0}^\infty \dual{c_n, d_n}_{{}_0H^\gamma(a,b)}
\]
for all $ \sum_{n=0}^\infty c_n \phi_n \in {}^0H^\gamma(a,b;\dot H^{-\beta}(\Omega))
$ and $ \sum_{n=0}^\infty d_n \phi_n \in {}_0H^\gamma(a,b;\dot H^\beta(\Omega)) $.
Since $ {}_0H^\gamma(a,b;\dot H^\beta(\Omega)) $ is reflexive, it is the dual space
of $ {}^0H^{-\gamma}(a,b;\dot H^{-\beta}(\Omega)) $. Above and throughout, for any
Banach space $ W $, the notation $ \dual{\cdot,\cdot}_W $ means the duality paring
between $ W^* $ (the dual space of $ W $) and W.

\subsection{Fractional calculus operators}
This section introduces fractional calculus operators on a domain $ (a,b) $, $
-\infty < a < b < \infty $, and summarizes several properties of these operators used
in the this paper. Assume that $ X $ is a separable Hilbert space.
\begin{Def}
	\label{def:frac_calc}
	For $ -\infty < \gamma < 0 $, define
	\begin{align*}
		\left(\D_{a+}^\gamma v\right)(t) &:=
		\frac1{ \Gamma(-\gamma) }
		\int_a^t (t-s)^{-\gamma-1} v(s) \, \mathrm{d}s, \quad t\in(a,b), \\
		\left(\D_{b-}^\gamma v\right)(t) &:=
		\frac1{ \Gamma(-\gamma) }
		\int_t^b (s-t)^{-\gamma-1} v(s) \, \mathrm{d}s, \quad t\in(a,b),
	\end{align*}
	for all $ v \in L^1(a,b;X) $, where $ \Gamma(\cdot) $ is the gamma function. In
	addition, let $ \D_{a+}^0 $ and $ \D_{b-}^0 $ be the identity operator on $
	L^1(a,b;X) $. For $ j - 1 < \gamma \leqslant j $ with $ j \in \mathbb N_{>0} $,
	define
	\begin{align*}
		\D_{a+}^\gamma v & := \D^j \D_{a+}^{\gamma-j}v, \\
		\D_{b-}^\gamma v & := (-\D)^j \D_{b-}^{\gamma-j}v,
	\end{align*}
	for all $ v \in L^1(a,b;X) $, where $ \D $ is the first-order differential operator
	in the distribution sense.
\end{Def}

Let $ \{e_n:n \in \mathbb N\} $ be an orthonormal basis of $ X $. For any $ \beta \in
\mathbb R $, define
\begin{small}
\begin{align*}
	{}^0H^\beta(a,b;X) := \bigg\{
		\sum_{n=0}^\infty c_n e_n:\
		\sum_{n=0}^\infty \nm{c_n}_{{}^0H^\beta(a,b)} < \infty
	\bigg\}
\end{align*}
\end{small}
and endow this space with the norm
\begin{small}
\[
	\Big\|\sum_{n=0}^\infty c_n e_n\Big\|_{{}^0H^\beta(a,b;X)} :=
	\bigg( \sum_{n=0}^\infty \nm{c_n}_{{}^0H^\beta(a,b)}^2 \bigg)^{1/2}.
\]
The space $ {}_0H^\beta(a,b;X) $ is analogously defined. It is standard that $
{}_0H^{-\beta}(a,b;X) $ is the dual space of $ {}^0H^\beta(a,b;X) $ in the sense that
\[
	\Dual{
		\sum_{n=0}^\infty c_n e_n,
		\sum_{n=0}^\infty d_n e_n
	}_{{}^0H^\beta(a,b;X)} :=
	\sum_{n=0}^\infty \dual{c_n, d_n}_{{}^0H^\beta(a,b)}
\]
\end{small}
for all $ \sum_{n=0}^\infty c_n e_n \in {}_0H^{-\beta}(a,b;X) $ and $
\sum_{n=0}^\infty d_n e_n \in {}^0H^\beta(a,b;X) $.
\begin{rem}
	\label{rem:equiv_frac_space}
	For any $ 0 < \beta < 1 $, a simple calculation gives that $ {}_0H^\beta(a,b;X) $
	is identical to $ [L^2(a,b;X), {}_0H^1(a,b;X)]_{\beta,2} $, and
	\[
		\nm{v}_{{}_0H^\beta(a,b;X)} \leqslant
		\sqrt2\,\nm{v}_{[L^2(a,b;X), {}_0H^1(a,b;X)]_{\beta,2}} \leqslant
		2\nm{v}_{{}_0H^\beta(a,b;X)}
	\]
	for all $ v \in {}_0H^\beta(a,b;X) $.
\end{rem}
\begin{lem}
	\label{lem:regu-basic}
	If $ 0 \leqslant \beta < \infty $ and $ -\infty < \gamma \leqslant \beta $, then
	\begin{small}
	\begin{align*}
		C_1 \nm{v}_{{}_0H^\beta(a,b;X)} \leqslant
		\nm{\D_{a+}^\gamma v}_{{}_0H^{\beta-\gamma}(a,b;X )} \leqslant
		C_2 \nm{v}_{{}_0H^\beta(a,b;X)}
		\,\forall v \in {}_0H^\beta(a,b;X), \\
		C_1 \nm{v}_{{}^0H^\beta(a,b;X)} \leqslant
		\nm{\D_{b-}^\gamma v}_{{}^0H^{\beta-\gamma}(a,b;X )} \leqslant
		C_2 \nm{v}_{{}^0H^\beta(a,b;X)}
		\,\forall v \in {}^0H^\beta(a,b;X),
	\end{align*}
	\end{small}
	where $ C_1 $ and $ C_2 $ are two positive constants depending only on $ \beta $
	and $ \gamma $.
\end{lem}
\begin{lem}
	\label{lem:coer}
	If $ -1/2 < \gamma < 1/2 $, then
	\begin{align*}
		\cos(\gamma\pi) \nm{\D_{a+}^\gamma v}_{L^2(a,b;X)}^2 \leqslant
		(\D_{a+}^\gamma v, \D_{b-}^\gamma v)_{L^2(a,b;X)} \leqslant
		\sec(\gamma\pi) \nm{\D_{a+}^\gamma v}_{L^2(a,b;X)}^2, \\
		\cos(\gamma\pi) \nm{\D_{b-}^\gamma v}_{L^2(a,b;X)}^2 \leqslant
		(\D_{a+}^\gamma v, \D_{b-}^\gamma v)_{L^2(a,b;X)} \leqslant
		\sec(\gamma\pi) \nm{\D_{b-}^\gamma v}_{L^2(a,b;X)}^2,
	\end{align*}
	for all $ v \in {}_0H^\gamma(a,b;X) $ (equivalent to $ {}^0H^\gamma(a,b;X) $),
	where $ (\cdot,\cdot)_{L^2(a,b;X)} $ is the usual inner product in $ L^2(a,b;X) $.
\end{lem}

By \cref{lem:regu-basic}, we can extend the domain of $ \D_{a+}^\gamma $, $ -\infty <
\gamma < 0 $, as follows. Assume that $ v \in {}_0H^\beta(a,b;X) $ with $ -\infty <
\beta < 0 $. If $ \beta \leqslant \gamma $, then define $ \D_{a+}^\gamma v \in
{}_0H^{\beta-\gamma}(a,b;X) $ by that
\begin{equation}
	\label{eq:extended_frac_int}
	\dual{\D_{a+}^\gamma v, w}_{{}^0H^{\gamma-\beta}(a,b;X)} :=
	\dual{v, \D_{b-}^\gamma w}_{{}^0H^{-\beta}(a,b;X)}
\end{equation}
for all $ w \in {}^0H^{\gamma-\beta}(a,b;X) $. If $ \beta > \gamma $, then define $
\D_{a+}^\gamma v \in {}_0H^{\beta-\gamma}(a,b;X) $ by that $ \D_{a+}^\gamma v =
\D_{a+}^{\gamma-\beta} \D_{a+}^\beta v $. The domain of the operator $ \D_{b-}^\gamma
$ can be extended analogously.

\begin{lem}
	\label{lem:regu}
	If $ -\infty < \beta < \infty $ and $ -\infty < \gamma \leqslant \max\{0,\beta\} $,
	then
	\begin{align*}
		C_1 \nm{v}_{{}_0H^\beta(a,b;X)} \leqslant
		\nm{\D_{a+}^\gamma v}_{{}_0H^{\beta-\gamma}(a,b;X )} \leqslant
		C_2 \nm{v}_{{}_0H^\beta(a,b;X)}
		\,\forall v \in {}_0H^\beta(a,b;X), \\
		C_1 \nm{v}_{{}^0H^\beta(a,b;X)} \leqslant
		\nm{\D_{b-}^\gamma v}_{{}^0H^{\beta-\gamma}(a,b;X )} \leqslant
		C_2 \nm{v}_{{}^0H^\beta(a,b;X)}
		\,\forall v \in {}^0H^\beta(a,b;X),
	\end{align*}
	where $ C_1 $ and $ C_2 $ are two positive constants depending only on $ \beta $
	and $ \gamma $.
\end{lem}

\begin{lem}
	If $ -\infty < \beta < \gamma < \beta+1/2 $, then
	\begin{equation}
		\dual{\D_{a+}^\gamma v, w}_{{}^0H^{\gamma-\beta}(a,b;X)} =
		\dual{\D_{a+}^\beta v, \D_{b-}^{\gamma-\beta} w}_{(a,b;X)}
	\end{equation}
	for all $ v \in {}_0H^\beta(a,b;X) $ and $ w \in {}^0H^{\gamma-\beta}(a,b;X) $.
\end{lem}

\begin{rem}
	For the proofs of the above lemmas, we refer the reader to \cite[Section
	3]{Luo2018Convergence}.
\end{rem}

\subsection{Algorithm definition}
Given $ J \in \mathbb N_{>0} $, set $ \tau := T/J $ and $ t_j := j\tau $, $ 0
\leqslant j \leqslant J $, and we use $ I_j $ to denote the interval $ (t_{j-1},t_j)
$ for each $ 1 \leqslant j \leqslant J $. Let $ \mathcal K_h $ be a shape-regular
triangulation of $\Omega $ consisting of $ d $-simplexes, and we use $ h $ to denote
the maximum diameter of the elements in $ \mathcal K_h $. Define
\begin{align*}
	%S_h & := \left\{
	%v_h \in \dot H^1(\Omega):\
	%v_h|_K \in P_1(K),\ \forall \,K \in \mathcal K_h
	%\right\},\\
	S_h & := \Big\{
		v_h \in \dot H^1(\Omega):\
		v_h \text{ is linear on each } K \in \mathcal K_h
	\Big\},\\
	W_{\tau,h} &:= \Big\{
		V \in L^2(0,T;S_h):\ V \text{ is constant on } I_j,\,
		1 \leqslant j \leqslant J
	\Big\}.
\end{align*}
For any $ V \in W_{\tau,h} $, we set
\begin{align*}
	V_j &:= \lim\limits_{t\to t_{j}-}V(t),
	\quad 1 \leqslant j \leqslant J, \\
	V^+_j &:= \lim\limits_{t\to t_{j}+}V(t),
	\quad 0 \leqslant j \leqslant J-1, \\
	\jmp{V_j} &:= V_j^{+} - V_j, \quad 0 \leqslant j \leqslant J,
\end{align*}
where the value of $ V_0 $ or $ V_J^{+} $ will be explicitly specified whenever
needed.

Assuming that $ u_0 \in S_h^* $ and $ f \in (W_{\tau,h})^* $, we define a numerical
solution $ U \in W_{\tau,h} $ to problem \cref{eq:model} by that $ U_0 = P_hu_0 $ and
\begin{equation}
	\label{eq:numer_sol}
	\sum_{j=0}^{J-1}\dual{\jmp{U_j},V^+_j}_\Omega +
	\dual{\nabla \D_{0+}^{-\alpha} U,\nabla V}_{\Omega \times (0,T)} =
	\dual{f,V}_{W_{\tau,h}}
\end{equation}
for all $ V \in W_{\tau,h} $, where $ P_h $ is the $ L^2 $-orthogonal projection onto
$ S_h $. Above and afterwards, for a Lebesgue measurable set $ \omega $ of $ \mathbb
R^l $ ($ l= 1,2,3,4 $), the symbol $ \dual{p,q}_\omega $ means $ \int_\omega pq $
whenever $ pq \in L^1(\omega) $. In addition, the symbol $ C_\times $ means a
positive constant depending only on its subscript(s), and its value may differ at
each occurrence.

\begin{thm}
	\label{thm:stab}
	Assume that $ u_0 \in L^2(\Omega) $. If $ f \in L^1(0,T;L^2(\Omega)) $, then
	\begin{equation}
		\label{eq:stab-1}
		\nm{U}_{L^\infty(0,T;L^2(\Omega))}
		\leqslant \sqrt2\ \nm{u_0}_{L^2(\Omega)} +
		2\nm{f}_{L^1(0,T;L^2(\Omega))}.
	\end{equation}
	If $ f \in {}_0H^{\alpha/2}(0,T;\dot H^{-1}(\Omega)) $, then
	\begin{equation}
		\label{eq:stab-2}
		\begin{aligned}
			& \nm{U}_{L^\infty(0,T;L^2(\Omega))} +
			\nm{U}_{{}_0H^{-\alpha/2}(0,T;\dot H^1(\Omega))} \\
			\leqslant{} &
			C_\alpha \big(
				\nm{u_0}_{L^2(\Omega)} +
				\nm{f}_{{}_0H^{\alpha/2}(0,T;\dot H^{-1}(\Omega))}
			\big).
		\end{aligned}
	\end{equation}
\end{thm}
\noindent For the proof of \cref{eq:stab-1}, we refer the reader to \cite[Theorem
2.1]{Mustapha2009Discontinuous}. By the techniques used in the proof of
\cref{thm:conv_f_L2} (in \cref{sec:main}), the proof of \cref{eq:stab-2} is trivial and
hence omitted.

\subsection{Weak solution and regularity}
\label{ssec:regu}
Following \cite{Luo2018Convergence}, we introduce the weak solution to problem
\cref{eq:model} as follows. Define
\begin{align*}
	W & := {}_0H^{(\alpha+1)/4}(0,T;L^2(\Omega))
	\cap {}_0H^{-(\alpha+1)/4}(0,T;\dot H^1(\Omega)), \\
	\widehat W & := {}^0H^{(3-\alpha)/4}(0,T;L^2(\Omega))
	\cap {}^0H^{(1-3\alpha)/4}(0,T;\dot H^1(\Omega)),
\end{align*}
and endow them with the norms
\begin{align*}
	\nm{\cdot}_W &:= \max\left\{
		\nm{\cdot}_{{}_0H^{(\alpha+1)/4}(0,T;L^2(\Omega)) },\
		\nm{\cdot}_{{}_0H^{-(\alpha+1)/4}(0,T;\dot H^1(\Omega))}
	\right\}, \\
	\nm{\cdot}_{\widehat W} &:= \max\left\{
		\nm{\cdot}_{{}^0H^{(3-\alpha)/4}(0,T;L^2(\Omega)) },\
		\nm{\cdot}_{{}^0H^{(1-3\alpha)/4}(0,T;\dot H^1(\Omega))}
	\right\},
\end{align*}
respectively. Assuming that $ u_0 t^{-(\alpha+1)/2} \in W^* $ and $ f \in \widehat
W^* $, we call $ u \in W $ a weak solution to problem \cref{eq:model} if
\begin{equation}
	\label{eq:weak_sol_f}
	\begin{aligned}
		{}& \Dual{
			\D_{0+}^{(\alpha\!+\!1)/2} u, v
		}_{ {}^0\!H^{(\alpha\!+\!1)/4}(0,T;L^2(\Omega)\!) } \!+\!
		\Dual{
			\nabla\! \D_{0+}^{\!-(\alpha\!+\!1)/4}u,
			\nabla\! \D_{T-}^{\!-(\alpha\!+\!1)/4} v
		}_{ \Omega\times(0,T)} \\
		={}&
		\Dual{ f,\ \D_{T-}^{(\alpha-1)/2} v }_{\widehat W} +
		\Dual{\frac{t^{-(\alpha+1)/2}}{\Gamma((1-\alpha)/2)} u_0,\ v}_W
	\end{aligned}
\end{equation}
for all $ v \in W $. In the above definition we have used the fact that, by
\cref{lem:regu-basic,lem:coer},
\[
	{}^0H^{(\alpha+1)/4}(0,T;L^2(\Omega)) =
	{}_0H^{(\alpha+1)/4}(0,T;L^2(\Omega))
	\quad\text{with equivalent norms,}
\]
and
\[
	{}^0H^{-(\alpha+1)/4}(0,T;\dot H^1(\Omega)) =
	{}_0H^{-(\alpha+1)/4}(0,T;\dot H^1(\Omega))
	\quad\text{with equivalent norms.}
\]
By the well-known Lax-Milgram theorem and \cref{lem:coer,lem:coer,lem:regu}, a
routine argument yields that the above weak solution is well-defined and admits the
stability estimate
\[
	\nm{u}_W \leqslant C_\alpha \Big(
		\nm{f}_{\widehat W^*} +
		\nm{t^{-(\alpha+1)/2} u_0}_{W^*}
	\Big).
\]
Furthermore, by a trivial modification of the proof of \cite[Theorems
4.2]{Luo2018Convergence}, we readily obtain the following regularity results.
\begin{thm}
	\label{thm:regu-pde}
	If $ u_0 = 0 $ and $ f \in {}_0H^\gamma(0,T;\dot H^\beta(\Omega)) $ with $
	(\alpha-3)/4 \leqslant \gamma < \infty $ and $ 0 \leqslant \beta < \infty $, then
	the solution $ u $ to problem \cref{eq:weak_sol_f} satisfies that
	\begin{align}
		& \D_{0+}^{\gamma+1} u - \Delta \D_{0+}^{\gamma-\alpha} u =
		\D_{0+}^\gamma f,
		\label{eq:strong-form} \\
		& \nm{u}_{{}_0H^{\gamma+1}(0,T;\dot H^\beta(\Omega))} +
		\nm{u}_{{}_0H^{\gamma-\alpha}(0,T;\dot H^{2+\beta}(\Omega))} \leqslant
		C_{\alpha,\gamma} \nm{f}_{{}_0H^\gamma(0,T;\dot H^\beta(\Omega))}.
		\label{eq:regu-pde}
	\end{align}
	Moreover, if $ 0 \leqslant \gamma < \alpha+1/2 $ then
	\begin{equation}
		\label{eq:regu-pde-C}
		\nm{u}_{C([0,T];\dot H^{\beta+(2\gamma+1)/(\alpha+1)}(\Omega))}
		\leqslant C_{\alpha,\gamma}
		\nm{f}_{{}_0H^\gamma(0,T;\dot H^\beta(\Omega))},
	\end{equation}
	and if $ \gamma = \alpha + 1/2 $ then
	\begin{equation}
		\nm{u}_{
			C([0,T];\dot H^{\beta+2(1-\epsilon)}(\Omega))
		} \leqslant \frac{C_\alpha}{\sqrt\epsilon}
		\nm{f}_{{}_0H^{\alpha+1/2}(0,T;\dot H^\beta(\Omega))}
	\end{equation}
	for all $ 0 < \epsilon < 1 $.
\end{thm}
%\begin{rem}
%Here we have used the fact that $ {}_H^{(\alpha+1)/4}(0,T;L^2(\Omega)) $ is
%equivalent to $ {}^0H^{(\alpha+1)/4}(0,T;L^2(\Omega)) $,
%\end{rem}
\begin{rem}
	For any $ v \in W $, since \cite[Lemma 33.2]{Tartar2007} implies
	\begin{small}
	\[
		\sqrt{
			\int_0^T t^{-(\alpha+1)/2} \nm{v(t)}_{L^2(\Omega)}^2
			\, \mathrm{d}t
		}\, \leqslant C_\alpha
		\nm{v}_{{}_0H^{(\alpha+1)/4}(0,T;L^2(\Omega))}
		\leqslant C_\alpha \nm{v}_W,
	\]
	\end{small}
	we have
	\begin{small}
	\begin{align*}
		& \Snm{
			\int_0^T t^{-(\alpha+1)/2}
			\dual{u_0, v(t)}_\Omega \, \mathrm{d}t
		} \\
		\leqslant{} &
		\nm{u_0}_{L^2(\Omega)} \sqrt{
			\int_0^T t^{-(\alpha+1)/2}dt \, \mathrm{d}t
		} \,\, \sqrt{
			\int_0^T t^{-(\alpha+1)/2} \nm{v(t)}_{L^2(\Omega)}^2 \, \mathrm{d}t
		} \\
		\leqslant{} &
		C_\alpha \nm{u_0}_{L^2(\Omega)} \nm{v}_W.
	\end{align*}
	\end{small}
	Therefore, $ t^{-(\alpha+1)/2} u_0 \in W^* $ and hence the above weak solution is
	well-defined for the case $ u_0 \in L^2(\Omega) $.
\end{rem}

Next, we briefly summarize two other methods to define the weak solution to problem
\cref{eq:model}. The first method uses the Mittag-Leffler function to define the weak
solution to problem \cref{eq:model} with $ f = 0 $ and $ u_0 \in \dot H^{r}(\Omega)
$, $ r \in \mathbb R $, by that \cite{McLean2007}
\[
	u(t) = \sum_{n=0}^{\infty}
	\dual{u_0,\phi_n}_{\dot H^{-r}(\Omega)}
	E_{\alpha+1,1}\big(-\lambda_n t^{\alpha+1}\big) \phi_n,
	\quad 0 \leqslant t \leqslant T,
\]
where, for any $ \beta, \gamma > 0 $, the Mittag-Leffler function $
E_{\beta,\gamma} $ is defined by
\[
	E_{\beta,\gamma}(z) := \sum_{n=0}^\infty
	\frac{z^n}{\Gamma(n\beta+\gamma)},
	\quad z \in \mathbb C.
\]
Then we can investigate the regularity of this weak solution by a growth estimate
\cite{Podlubny1998}: for any $ \beta,\gamma,t >0 $,
\[
	\snm{E_{\beta,\gamma}(-t)} \leqslant{}
	\frac{C_{\beta,\gamma}}{1+t}.
\]
%\begin{thm}
%If $ u_0 = 0 $ and $ f(t) = v \in \dot H^\beta(\Omega) $ for each $ 0 < t \leqslant
%T $, $ 0 \leqslant \beta < \infty $, then
%\begin{equation}
%\nm{u}_{C([0,T];\dot H^{\beta+1/(\alpha+1)}(\Omega))} \leqslant
%C \nm{v}_{\dot H^\beta(\Omega)}.
%\end{equation}
%\begin{equation}
%\nm{u'}_{L^2(0,T;\dot H^1(\Omega))} \leqslant
%C_\alpha \nm{v}_{\dot H^{\alpha/(1+\alpha)}(\Omega)}
%\end{equation}
%If $ v \in \dot H^{2\alpha/(\alpha+1)}(\Omega) $, then
%\begin{equation}
%\nm{u}_{C([0,T];\dot H^2(\Omega))} \leqslant
%C_\alpha \nm{v}_{\dot H^{2\alpha/(\alpha+1)}(\Omega)}.
%\end{equation}
%\end{thm}
%\begin{proof}
%We have
%\begin{equation}
%u(t) = \sum_{n=0}^\infty \dual{v,\phi_n}_\Omega
%t E_{1+\alpha,2}(-\lambda_n t^{1+\alpha}) \phi_n,
%\end{equation}
%\end{proof}

The second method uses the well-known transposition technique to define the weak
solution to problem \cref{eq:model} as follows. Define
\[
	G := {}^0H^1(0,T;L^2(\Omega)) \cap {}^0H^{-\alpha}(0,T;\dot H^2(\Omega)),
\]
and equip this space with the norm
\[
	\nm{\cdot}_G := \max\left\{
		\nm{\cdot}_{{}^0H^1(0,T;L^2(\Omega))},\
		\nm{\cdot}_{{}^0H^{-\alpha}(0,T;\dot H^2(\Omega))}
	\right\}.
\]
Also, define
\[
	G_\mathrm{tr} := \big\{ v(0):\ v \in G \big\},
\]
and endow this space with the norm
\[
	\nm{v_0}_{G_\mathrm{tr}} :=
	\inf_{v \in G,\ v(0) = v_0} \nm{v}_G
	\quad \forall v_0 \in G_\mathrm{tr}.
\]
Assuming that $ u_0 \in G_\mathrm{tr}^* $ and $ f \in G^* $, we call $ u $ a weak
solution to problem \cref{eq:model} if
\[
	\dual{u, -v' - \Delta \D_{T-}^{-\alpha} v}_{\Omega \times (0,T)} =
	\dual{f, v}_G + \dual{u_0, v(0)}_{G_\mathrm{tr}}
\]
for all $ v \in G $. By the symmetric version of \cref{thm:regu-pde}, applying the
famous Babu\v{s}ka-Lax-Milgram thoerem proves that the above weak solution is
well-defined.

\section{Discretizations of two fractional ordinary equations}
\label{sec:disc_regu}
\subsection{An auxiliary function}
For any $ z \in \{x+iy:\ 0 < x < \infty, -\infty < y < \infty \} $, define
\begin{equation}
	\label{eq:def-psi}
	\psi(z) := \frac{e^z-1}{\Gamma(2+\alpha)}
	\sum_{k=1}^\infty k^{1+\alpha} e^{-kz}.
\end{equation}
By the standard analytic continuation technique, $ \psi $ has a Hankel integral
representation (cf.~\cite[(12.1)]{Wood1992} and \cite[(21)]{McLean2015Time})
\begin{align*}
	\psi(z) & = \frac{e^z-1}{2\pi i}
	\int_{-\infty}^{({0+})} \frac{w^{-2-\alpha}}{e^{z-w}-1} \, \mathrm{d}w,
	\quad z \in \mathbb C \setminus (-\infty,0],
\end{align*}
where $ \int_{-\infty}^{({0+})} $ means an integral on a piecewise smooth and
non-self-intersecting path enclosing the negative real axis and orienting
counterclockwise, $ 0 $ and $ \{z+2k\pi i \neq 0: k \in \mathbb Z\} $ lie on the
different sides of this path, and $ w^{-2-\alpha} $ is evaluated in the sense that
\[
	w^{-2-\alpha} = e^{-(2+\alpha) \operatorname{Log}w}.
\]
By Cauchy's integral theorem and Cauchy's integral formula, it is clear that
(cf.~\cite[(13.1)]{Wood1992})
\begin{equation}
	\label{eq:psi}
	\psi(z) = (e^z-1) \sum_{k \in \mathbb Z} (z+2k\pi i)^{-2-\alpha},
\end{equation}
for all $ z \in \mathbb C \setminus (-\infty, 0] $ satisfying $ -2\pi <
\operatorname{Im} z < 2\pi $. From this series representation, it follows that
\begin{equation}
	\label{eq:psi-conj}
	\psi(z) = \overline{\psi(\overline z)}
	\quad \text{for all}\, z \in \mathbb C \setminus (-\infty,0]
	\text{ with } \snm{\operatorname{Im}z} < 2\pi.
\end{equation}
Moreover,
\begin{equation}
	\label{eq:shit-7}
	\psi(z) - (e^z-1)z^{-2-\alpha} \quad
	\text{is analytic on }
	\{w \in \mathbb C:\ \snm{\operatorname{Im} w} < 2\pi\},
\end{equation}
and hence
\begin{equation}
	\label{eq:psi-singu}
	\begin{aligned}
		& \lim_{r \to 0+} \frac{\psi(re^{i\theta})}{
			r^{-1-\alpha}\big(
				\cos((1+\alpha)\theta) - i \sin((1+\alpha)\theta)
			\big)
		} = 1 \\
		& \text{ uniformly for all }  -\pi < \theta < \pi.
	\end{aligned}
\end{equation}

\begin{lem}
	\label{lem:1+mupsi}
	There exist $ \pi/2 < \theta_\alpha \leqslant (\alpha+3)/(4\alpha+4)\pi $,
	depending only on $ \alpha $, and $ 0 < \delta_{\alpha,\mu} < \infty $, depending
	only on $ \alpha $ and $ \mu $, such that
	\begin{equation}
		\label{eq:1+mupsi}
		\begin{aligned}
			& 1 + \mu \psi(z) \neq 0 \quad\text{ for all }
			z  \in \big\{
				w \in \mathbb C \setminus \{0\}:
				-\pi \leqslant \operatorname{Im} w \leqslant \pi
			\big\} \bigcap {} \\
			& \qquad\qquad\qquad \big\{
				w \in \mathbb C:\
				0 < \operatorname{Re} w \leqslant \delta_{\alpha,\mu}
				\text{ or } \pi/2 \leqslant
				\snm{\operatorname{Arg} w} \leqslant \theta_\alpha
			\big\},
		\end{aligned}
	\end{equation}
	where $ \mu $ is a nonnegative constant.
\end{lem}
\begin{proof}
	By \cref{eq:psi-singu}, there exists $ 0 < \delta_\alpha < \pi $, depending only on
	$ \alpha $, such that $ \operatorname{Im} \psi(z) < 0 $ and hence
	\begin{small}
	\begin{equation}
		\label{eq:shit-4}
		1+\mu\psi(z) \neq 0 \text{ for all }
		z \in \Big\{
			w \in \mathbb C:\ \pi/2
			\leqslant \operatorname{Arg} w \leqslant
			\frac{\alpha+3}{4(\alpha+1)}\pi,\
			0 < \operatorname{Im} w \leqslant \delta_\alpha
		\Big\}.
	\end{equation}
	\end{small}
	For $ 0 < y \leqslant \pi $, by \cref{eq:psi} we have
	\begin{align}
		\psi(iy) & = (e^{iy}-1) \sum_{k=-\infty}^\infty
		(iy+2k\pi i)^{-2-\alpha} \notag \\
		& =
		(e^{iy}-1) \Big(
			\sum_{k=-\infty}^{-1}  (-2k\pi - y)^{-2-\alpha}
			(-i)^{-2-\alpha} +
			\sum_{k=0}^\infty (2k\pi+y)^{-2-\alpha} i^{-2-\alpha}
		\Big) \notag \\
		&=
		(1-e^{iy})\Big(
			\sum_{k=1}^\infty (2k\pi - y)^{-2-\alpha}
			e^{i\alpha\pi/2} + \sum_{k=0}^\infty (2k\pi+y)^{-2-\alpha} e^{-i\alpha\pi/2}
		\Big) \notag \\
		&= (1-e^{iy}) (A + iB), \label{eq:psi-A-B}
	\end{align}
	where
	\begin{align*}
		A &:=  \cos(\alpha\pi/2)
		\sum_{k=0}^\infty \Big(
			(2k\pi + 2\pi - y)^{-2-\alpha}  +
			(2k\pi+y)^{-2-\alpha}
		\Big), \\
		B &:= \sin(\alpha\pi/2) \sum_{k=0}^\infty
		\Big(
			(2k\pi+2\pi-y)^{-2-\alpha} -
			(2k\pi+y)^{-2-\alpha}
		\Big).
	\end{align*}
	It follows that
	\begin{align*}
		\operatorname{Re} \psi(iy) = A (1-\cos y) + B \sin y, \\
		\operatorname{Im} \psi(iy) = B(1-\cos y) - A \sin y.
	\end{align*}
	%Therefore, a straightforward computation proves \cref{eq:psi-near-0}, and
	A straightforward computation then gives
	\begin{align}
		\operatorname{Re} \psi(i\pi) =  4\pi^{-2-\alpha} \cos(\alpha\pi/2)
		\sum_{k=1}^\infty (2k-1)^{-2-\alpha} > 0, \label{eq:psipi+} \\
		\operatorname{Im} \psi(iy) < 0, \quad 0 < y < \pi, \label{eq:psipi+2}
	\end{align}
	and hence, by the continuity of $ \psi $ in
	\[
		\left\{
			z \in \mathbb C \setminus (-\infty,0]:
			-2\pi < \operatorname{Im} \psi(z) < 2\pi
		\right\},
	\]
	a routine argument yields that there exists $ 0 < r_\alpha \leqslant
	\delta_\alpha\tan((1-\alpha)/(4\alpha+4)\pi) $, depending only on $ \alpha $, such
	that
	\begin{equation}
		\label{eq:426-2}
		1+\mu\psi(z) \neq 0 \text{ for all }
		z \in \left\{
			w \in \mathbb C:\
			-r_\alpha \leqslant \operatorname{Re} w \leqslant 0,\
			\delta_\alpha \leqslant \operatorname{Im} w \leqslant \pi
		\right\}.
	\end{equation}
	By \cref{eq:shit-4,eq:426-2}, letting $ \theta_\alpha := \pi/2 +
	\operatorname{arctan}(r_\alpha/\pi) $ yields
	\begin{equation}
		\label{eq:shit-10}
		1 + \mu\psi(z) \neq 0 \text{ for all }
		z \in \{w \in \mathbb C:\
		\pi/2 \leqslant \operatorname{Arg}w \leqslant \theta_\alpha,\
		0 < \operatorname{Im} w \leqslant \pi\}.
	\end{equation}
	In addition, by \cref{eq:psipi+}, \cref{eq:psipi+2}, \cref{eq:psi-singu} and the
	continuity of $ \psi $ in
	\[
		\left\{
			z \in \mathbb C \setminus (-\infty,0]:\
			-2\pi < \operatorname{Im} z < 2\pi
		\right\},
	\]
	there exists $ \delta_{\alpha,\mu} > 0 $ depending only on $ \alpha $ and $ \mu $
	such that
	\begin{equation}
		\label{eq:426-3}
		1+\mu\psi(z) \neq 0 \text{ for all }
		z \in \{
			w \in \mathbb C \setminus \{0\}:\
			0 \leqslant \operatorname{Re} w \leqslant \delta_{\alpha,\mu},\
			0 \leqslant \operatorname{Im} w \leqslant \pi
		\}.
	\end{equation}
	Finally, by \cref{eq:psi-conj}, combining \cref{eq:shit-10,eq:426-3} proves
	\cref{eq:1+mupsi} and hence this lemma.
\end{proof}

\begin{lem}
	\label{lem:psi-growth}
	For any $ \mu > 0 $ and $ 0 < y \leqslant \pi $,
	\begin{equation}
		\label{eq:psi-growth}
		\snm{1+\mu\psi(iy)} > C_\alpha(1+\mu y^{-1-\alpha}).
	\end{equation}
\end{lem}
\begin{proof}
	By \cref{eq:psi-singu,eq:psipi+,eq:psipi+2}, there exists $ 0 < y_\alpha < \pi $,
	depending only on $ \alpha $, such that
	\begin{align*}
		\operatorname{Re} \psi(iy) > C_\alpha y^{-1-\alpha}
		\quad\forall\, y_\alpha \leqslant y \leqslant \pi, \\
		\operatorname{Im} \psi(iy) < -C_\alpha y^{-1-\alpha}
		\quad\forall\, 0 < y \leqslant y_\alpha.
	\end{align*}
	It follows that
	\[
		\snm{1+\mu\psi(iy)} > C_\alpha \mu y^{-1-\alpha}
		\quad\forall\, 0 < y \leqslant \pi,
	\]
	and hence
	\begin{align*}
		\inf_{
			\substack{
				0 < y \leqslant \pi \\
				y^{1+\alpha} \leqslant \mu < \infty
			}
		} \frac{\snm{1+\mu\psi(iy)}}{1+\mu y^{-1-\alpha}} \geqslant
		\inf_{
			\substack{
				0 < y \leqslant \pi \\
				y^{1+\alpha} \leqslant \mu < \infty
			}
		} \frac{y^{1+\alpha}}{2\mu} \snm{1+\mu\psi(iy)} > C_\alpha.
		%} y^{1+\alpha}/(2\mu) \snm{1+\mu\psi(iy)} > C_\alpha.
	\end{align*}
	It remains therefore to prove
	\begin{equation}
		\label{eq:psi-growth-1}
		\inf_{
			\substack{
				0 < \mu \leqslant \pi^{1+\alpha} \\
				\mu^{1/(1+\alpha)} \leqslant y \leqslant \pi
			}
		} \frac{\snm{1+\mu\psi(iy)}}{1+\mu y^{-1-\alpha}}
		> C_\alpha.
	\end{equation}

	To this end, we proceed as follows. By \cref{eq:psi}, there exists a continuous
	function $ g $ on $ [0,\pi] $ such that $ g(0) = 0 $ and
	\[
		\psi(iy) = (iy)^{-1-\alpha} + y^{-1-\alpha} g(y),
		\quad 0 < y \leqslant \pi.
	\]
	A straightforward computation gives
	\begin{align*}
		& 2\snm{1+\mu\psi(iy)}^2 \\
		={} & 2\snm{
			1+\mu(iy)^{-1-\alpha} +
			\mu y^{-1-\alpha} g(y)
		}^2 \\
		\geqslant{} &
		\snm{1+\mu(iy)^{-1-\alpha}}^2 -
		2\mu^2 y^{-2-2\alpha} \snm{g(y)}^2 \\
		={} &
		1 + \mu^2 y^{-2-2\alpha} +
		2\mu y^{-1-\alpha} \cos((1+\alpha)\pi/2) -
		2\mu^2 y^{-2-2\alpha} \snm{g(y)}^2 \\
		={} &
		\Big( \mu y^{-1-\alpha} + \cos^2\big((1+\alpha)\pi/2\big) \Big) +
		\sin^2\big((1+\alpha)\pi/2\big) -
		2\mu^2 y^{-2-2\alpha}\snm{g(y)}^2 \\
		\geqslant{} &
		\sin^2\big((1+\alpha)\pi/2\big) -
		2\mu^2 y^{-2-2\alpha}\snm{g(y)}^2,
	\end{align*}
	so that, by the fact $ g(0) = 0 $, there exists $ 0 < y_\alpha < \pi $, depending
	only on $ \alpha $, such that
	\[
		\inf_{
			\substack{
				0 < \mu \leqslant y_\alpha^{1+\alpha} \\
				\mu^{1/(1+\alpha)} \leqslant y \leqslant y_\alpha
			}
		} \snm{1+\mu\psi(iy)} > C_\alpha.
	\]
	In addition, applying the extreme value theorem yields, by \cref{eq:1+mupsi}, that
	\begin{equation*}
		\inf_{
			\substack{
				0 \leqslant \mu \leqslant \pi^{1+\alpha} \\
				y_\alpha \leqslant y \leqslant \pi
			}
		} \snm{1+\mu\psi(iy)} > C_\alpha.
	\end{equation*}
	Using the above two estimates yields \cref{eq:psi-growth-1}, by the estimate
	\[
		\inf_{
			\substack{
				0 < \mu \leqslant \pi^{1+\alpha} \\
				\mu^{1/(1+\alpha)} \leqslant y \leqslant \pi
			}
		} \frac{\snm{1+\mu\psi(iy)}}{1+\mu y^{-1-\alpha}}
		\geqslant \frac12 \inf_{
			\substack{
				0 < \mu \leqslant \pi^{1+\alpha} \\
				\mu^{1/(1+\alpha)} \leqslant y \leqslant \pi
			}
		}
		\snm{1+\mu\psi(iy)}.
	\]
	This completes the proof.
\end{proof}

\begin{lem}
	\label{lem:g'}
	For any $ \mu > 0 $ and $ 0 < y \leqslant \pi $,
	\begin{equation}
		\label{eq:g'}
		\snm{g'(y)} < C_\alpha
		\frac{\mu y^{-2-\alpha}}{(1+\mu y^{-1-\alpha})^2},
	\end{equation}
	where $ g(y) := (1+\mu\psi(iy))^{-1} $.
\end{lem}
\begin{proof}
	By \cref{eq:psi-A-B}, $ \psi(iy) $ can be expressed in the form
	\[
		\psi(iy) = F(y) + G(y), \quad 0 < y \leqslant \pi,
	\]
	where $ F $ is analytic on $ [0,\pi] $ and
	\[
		G(y) = (1-e^{iy}) y^{-2-\alpha}
		\big( \cos(\alpha\pi/2) - i \sin(\alpha\pi/2) \big).
	\]
	A direct calculation gives
	\[
		\snm{G'(y)} < C_\alpha y^{-2-\alpha},
		\quad 0 < y \leqslant \pi,
	\]
	so that
	\[
		\Snm{i\psi'(iy)} = \snm{F'(y) + G'(y)}
		< C_\alpha y^{-2-\alpha}, \quad 0 < y \leqslant \pi.
	\]
	In addition, \cref{lem:psi-growth} implies
	\[
		\snm{1+\mu\psi(iy)}^{-2} < C_\alpha(1+\mu y^{-1-\alpha})^{-2},
		\quad 0 < y \leqslant \pi.
	\]
	Therefore, \cref{eq:g'} follows from the equality
	\[
		g'(y) = \frac{i\mu \psi'(iy)}{(1+\mu\psi(iy))^2}.
	\]
	This completes the proof.
\end{proof}

In the next two subsections, we use $ \theta $ to abbreviate $ \theta_\alpha $,
defined in \cref{lem:1+mupsi}, define
\[
	\Upsilon := (\infty,0]e^{-i\theta} \cup [0,\infty)e^{i\theta},
\]
and let $ \Upsilon $ be oriented so that $ \operatorname{Im} z $ increases along $
\Upsilon $. In addition, $ \Upsilon_1 := \{z \in \Upsilon:\ \snm{\operatorname{Im} z}
\leqslant \pi\} $ and it inherits the orientation of $ \Upsilon $.

\subsection{The first fractional ordinary equation}
\label{ssec:first_ode}
This subsection considers the fractional ordinary equation
\begin{equation}
	\label{eq:ode-y}
	\xi'(t) + \lambda \D_{0+}^{-\alpha} \xi(t) = 0,
	\quad t > 0,
\end{equation}
subjected to the initial value condition $ \xi(0) = \xi_0 $, where $ \lambda $ is a
positive constant and $ \xi_0 \in \mathbb R $. By \cite[(2.1)]{Lubich1996}, the
solution $ \xi $ of equation \cref{eq:ode-y} is expressed by a contour integral
\begin{equation}
	\label{eq:y}
	\xi(t) = \frac{\xi_0}{2\pi i} \int_\Upsilon
	e^{tz} z^\alpha(z^{1+\alpha}+\lambda)^{-1} \, \mathrm{d}z,
	\quad t > 0.
\end{equation}
Applying the temporal discretization used in \cref{eq:numer_sol} to equation
\cref{eq:ode-y} yields the following discretization: let $ Y_0 = \xi_0 $; for $ k \in
\mathbb N $, the value of $ Y_{k+1} $ is determined by that
\[
	\mu \Big(
		\sum_{j=1}^k Y_j \big(
			b_{k-j+2} - 2b_{k-j+1} + b_{k-j}
		\big) + b_1 Y_{k+1}
	\Big) + Y_{k+1} - Y_k = 0,
\]
where $ \mu := \lambda \tau^{1+\alpha} $ and $ b_j := j^{1+\alpha}/\Gamma(2+\alpha)
$, $ j \in \mathbb N $.
%This discretization possesses the following stability and convergence estimates.
\begin{thm}
	\label{thm:Y-jump}
	For any $ k \in \mathbb N_{>0} $ we have
	\begin{equation}
		\label{eq:Y-jump}
		\snm{Y_{k+1} - Y_k} \leqslant C_\alpha k^{-1} \snm{\xi_0}.
	\end{equation}
\end{thm}
\begin{thm}
	\label{thm:y-Y}
	For any $ k \in \mathbb N_{>0} $ we have
	\begin{align}
		\snm{\xi(t_k) - Y_k} &\leqslant C_\alpha k^{-1} \snm{\xi_0}
		\label{eq:y-Y}.
		%\snm{y(t_k) - Y_k} &\leqslant C_\alpha \lambda^{1/(1+\alpha)} \tau.
		%\label{eq:y-Y-2}
	\end{align}
\end{thm}

The main task of the rest of this subsection is to prove the above two theorems by
the well-known Laplace transform method (the basic idea comes from
\cite{Lubich1996,McLean2015Time,Jin2015}). We introduce the discrete Laplace
transform of $ (Y_k)_{k=0}^\infty $ by that
\begin{equation}
	\label{eq:Y-laplace}
	\widehat Y(z) := \sum_{k=0}^\infty Y_k e^{-kz}
	\quad \forall z \in H,
\end{equation}
where $ H := \{x+iy: 0 < x \leqslant \delta_{\alpha,\mu},\, -\pi \leqslant y
\leqslant \pi \} $, with $ \delta_{\alpha,\mu} $ being defined in \cref{lem:1+mupsi}.
By the definition of the sequence $ (Y_k)_{k=0}^\infty $, a straightforward
computation gives
\begin{align*}
	\mu (\widehat Y(z)-\xi_0) (e^z-1)^2\, \widehat b(z) +
	(\widehat Y(z) - \xi_0) e^z - \widehat Y(z) = 0,
	\quad z \in H,
\end{align*}
where $ \widehat b $ is the discrete transform of the sequence $ (b_k)_{k=0}^\infty
$, namely,
\[
	\widehat b(z) = \sum_{k=1}^\infty
	\frac{k^{1+\alpha}}{\Gamma(2+\alpha)} e^{-kz}.
\]
For any $ z \in H $, combining like terms yields
\[
	(e^z-1 + \mu(e^z-1)^2 \widehat b(z)) \widehat Y(z) -
	\big( e^z + \mu(e^z-1)^2\,\widehat b(z) \big) \xi_0 = 0,
\]
so that
\begin{align*}
	\widehat Y(z) &=
	\frac{ e^z + \mu(e^z-1)^2\, \widehat b(z) }{
		e^z-1 + \mu(e^z-1)^2\,\widehat b
	} \, \xi_0 \\
	&= \Big(
		1 + \frac1{e^z-1 + \mu(e^z-1)^2\,\widehat b}
	\Big) \xi_0 \\
	&= \Big(
		1 + \frac{(e^z-1)^{-1}}{1+\mu\psi(z)}
	\Big) \xi_0,
\end{align*}
by \cref{eq:def-psi,lem:1+mupsi}. Therefore, a routine calculation
(cf.~\cite[(28)]{McLean2015Time}) yields that, for any $ 0 < a \leqslant
\delta_{\alpha,\mu} $ and $ k \in \mathbb N_{>0} $,
\begin{align*}
	Y_k & = \frac{\xi_0}{2\pi i} \int_{a-i\pi}^{a+i\pi}
	\widehat Y(z) e^{kz} \, \mathrm{d}z =
	\frac{\xi_0}{2\pi i} \int_{a-i\pi}^{a+i\pi}
	\frac{e^{kz}}{1+\mu\psi(z)} \frac{\mathrm{d}z}{e^z-1}.
	%& =
	%\frac{\xi_0}{2\pi} e^{ka} \int_{-\pi}^\pi
	%\frac{e^{iky}}{1+\mu\psi(a+iy)} \frac{\mathrm{d}y}{e^{a+iy}-1}.
\end{align*}
By \cref{eq:psi-singu,eq:1+mupsi}, letting $ a \to {0+} $ and applying Lebesgue's
dominated convergence theorem then yields
\begin{equation}
	\label{eq:Y_k}
	Y_k = \frac{\xi_0}{2\pi i}\,
	\int_{-i\pi}^{i\pi} \frac{e^{kz}}{1+\mu\psi(z)}
	\, \frac{\mathrm{d}z}{e^z-1}.
\end{equation}

By \cref{eq:1+mupsi} we have that the integrand in \cref{eq:Y_k} is analytic on
\[
	\omega := \left\{
		z \in \mathbb C:\
		0 < \snm{\operatorname{Im} z} < \pi,\
		\pi/2 < \snm{\operatorname{Arg} z} < \theta
	\right\},
\]
this integrand is continuous on $ \partial\omega \setminus \{0\} $, and
\cref{eq:psi-singu} implies that
\[
	\lim_{\omega \ni z \to 0} \snm{z}^{-\alpha}
	\snm{e^{kz}(1+\mu\psi(z))^{-1}(e^z-1)^{-1}} = \mu^{-1}.
\]
Additionally,
\[
	\frac{e^{kz}}{(1+\mu\psi(z))(e^z-1)} =
	\frac{e^{k(z+2\pi i)}}{(1+\mu\psi(z+2\pi i))(e^{z+2\pi i}-1)}
\]
for all $ z = x-i\pi $, $ -\pi \tan\theta \leqslant x \leqslant 0 $. Therefore, an
elementary calculation yields
\begin{equation}
	\label{eq:Y_k-3}
	Y_k = \frac{\xi_0}{2\pi i} \int_{\Upsilon_1}
	\frac{e^{kz}}{1+\mu\psi(z)} \frac{\mathrm{d}z}{e^z-1},
\end{equation}
by \cref{eq:Y_k} and Cauchy's integral theorem.

\begin{rem}
	By the techniques used in the proof of \cref{thm:stab}, it is easy to obtain that $
	\snm{Y_k} \leqslant \snm{\xi_0} $ for all $ k \in \mathbb N_{>0} $. Therefore, the
	series in \cref{eq:Y-laplace} converge absolutely for all $ z \in H $.
\end{rem}

Finally, we present the proofs of \cref{thm:Y-jump,thm:y-Y} as follows.

\medskip\noindent{\bf Proof of \cref{thm:Y-jump}.} Firstly, let us prove
	\begin{equation}
		\label{eq:cos-g}
		\Big\lvert \int_0^\pi \cos(ky) g(y) \, \mathrm{d}y \Big\rvert
		\leqslant C_\alpha k^{-1},
	\end{equation}
	where $ g(y) := (1+\mu\psi(iy))^{-1} $, $ 0 < y \leqslant \pi $. A straightforward
	computation gives
	\begin{align*}
		\int_0^\pi \cos(ky) g(y) \, \mathrm{d}y &=
		\sum_{j=1}^k \int_{(j-1)\pi/k}^{j\pi/k}
		\cos(ky) g(y) \, \mathrm{d}y \\
		&= \sum_{j=1}^k \int_{(j-1)\pi/k}^{j\pi/k}
		\cos(ky) \big( g(y)-g((j-1)\pi/k) \big) \, \mathrm{d}y \\
		&= \sum_{j=1}^k \int_{(j-1)\pi/k}^{j\pi/k}
		\cos(ky) \int_{(j-1)\pi/k}^y g'(s) \, \mathrm{d}s \, \mathrm{d}y.
	\end{align*}
	It follows that
	\begin{small}
	\begin{align*}
		& \Snm{\int_0^\pi \cos(ky) g(y) \, \mathrm{d}y}
		\leqslant \sum_{j=1}^k \int_{(j-1)\pi/k}^{j\pi/k}
		\int_{(j-1)\pi/k}^y \snm{g'(s)} \, \mathrm{d}s \, \mathrm{d}y \\
		<{} & \pi k^{-1} \int_0^\pi \snm{g'(y)} \, \mathrm{d}y
		< C_\alpha k^{-1} \int_0^\pi
		\frac{\mu y^{-2-\alpha}}{(1+\mu y^{-1-\alpha})^2} \, \mathrm{d}y
		\quad\text{(by \cref{lem:g'})} \\
		<{} & C_\alpha k^{-1} \bigg(
			\int_0^{\mu^{1/(1+\alpha)}}
			\frac{\mu y^{-2-\alpha}}{(1+\mu y^{-1-\alpha})^2} \, \mathrm{d}y +
			\int_{\mu^{1/(1+\alpha)}}^{\max\{\mu^{1/(1+\alpha)}, \pi\}}
			\frac{\mu y^{-2-\alpha}}{(1+\mu y^{-1-\alpha})^2} \, \mathrm{d}y
		\bigg) \\
		<{} &
		C_\alpha k^{-1} \bigg(
			\int_0^{\mu^{1/(1+\alpha)}}
			\mu^{-1} y^\alpha \, \mathrm{d}y +
			\int_{\mu^{1/(1+\alpha)}}^{\max\{\mu^{1/(1+\alpha)},\pi\}}
			\mu y^{-2-\alpha} \, \mathrm{d}y
		\bigg) < C_\alpha k^{-1},
	\end{align*}
	\end{small}
	which proves \cref{eq:cos-g}.

	Secondly, let us prove
	\begin{equation}
		\label{eq:sin-g}
		\Big\lvert \int_0^\pi \sin(ky) g(y) \, \mathrm{d}y \Big\rvert
		< C_\alpha k^{-1}.
	\end{equation}
	If $ k = 1 + 2m $, $ m \in \mathbb N $, then a similar argument as that to derive
	\cref{eq:cos-g} yields
	\begin{align*}
		\Big\lvert \int_0^{2m\pi/(1+2m)} \sin(ky) g(y) \, \mathrm{d}y \Big\rvert
		< C_\alpha k^{-1},
	\end{align*}
	and hence \cref{eq:sin-g} follows from the estimate
	\begin{align*}
		\Big\lvert \int_{2m\pi/(1+2m)}^\pi \sin(ky) g(y) \, \mathrm{d}y \Big\rvert
		< C_\alpha k^{-1},
	\end{align*}
	which is evident by \cref{lem:psi-growth}. If $ k = 2m $, $ m \in \mathbb N_{>0} $,
	then a simple modification of the above analysis proves that \cref{eq:sin-g} still
	holds.

	Finally, combining \cref{eq:cos-g,eq:sin-g} yields
	\[
		\Big\lvert \int_0^\pi e^{iky} g(y) dy \Big\rvert
		\leqslant C_\alpha k^{-1},
	\]
	so that
	\[
		\Bigl\lvert \operatorname{Re} \int_0^\pi e^{iky} g(y) \, \mathrm{d}y \Big\rvert
		\leqslant C_\alpha k^{-1}.
	\]
	Therefore, \cref{eq:Y-jump} follows from
	\[
		Y_{k+1} - Y_k = \frac{\xi_0}\pi\,
		\operatorname{Re} \int_0^\pi e^{iky}g(y) \, \mathrm{d}y,
	\]
	which is evident by \cref{eq:psi-conj,eq:Y_k}. This concludes the proof of
	\cref{thm:Y-jump}.
	%so that \cref{eq:psi-conj} implies
	%\begin{equation}
	%\label{eq:Y_k-2}
	%Y_k = \frac{\xi_0}\pi\,
	%\operatorname{Re} \int_0^\pi \frac{e^{iky}}{1+\mu\psi(iy)}
	%\, \frac{\mathrm{d}y}{e^{iy}-1}.
	%\end{equation}
\hfill\ensuremath{\blacksquare}

\medskip\noindent{\bf Proof of \cref{thm:y-Y}.} Substituting $ \eta:= \tau z $ into
	\cref{eq:y} yields
	\[
		\xi(t_k) = \frac{\xi_0}{2\pi i} \int_\Upsilon
		e^{k\eta}(\eta + \mu\eta^{-\alpha})^{-1}  \, \mathrm{d}\eta,
	\]
	and then subtracting \cref{eq:Y_k-3} from this equation gives
	\begin{equation}
		\label{eq:y_k-Y_k}
		\xi(t_k) - Y_k = \mathbb I_1 + \mathbb I_2,
	\end{equation}
	where
	\begin{align*}
		\mathbb I_1 &:= \frac{\xi_0}{2\pi i}
		\int_{\Upsilon \setminus \Upsilon_1} e^{kz}
		(z+\mu z^{-\alpha})^{-1} \, \mathrm{d}z, \\
		\mathbb I_2 &:=
		\frac{\xi_0}{2\pi i} \int_{\Upsilon_1} e^{kz} \Big(
			(z + \mu z^{-\alpha})^{-1} -
			(1+\mu\psi(z))^{-1} (e^z-1)^{-1}
		\Big) \, \mathrm{d}z.
	\end{align*}
	Since $\mathbb{I}_1$ is a real number, a simple calculation gives
	\begin{align*}
		\mathbb I_1 &=
		\frac{\xi_0}\pi
		\operatorname{Im} \int_{\pi/\sin\theta}^\infty
		e^{kre^{i\theta}} (re^{i\theta} + \mu (re^{i\theta})^{-\alpha})^{-1}
		e^{i\theta} \,\mathrm{d}r \\
		&= \frac{\xi_0}\pi
		\operatorname{Im} \int_{\pi/\sin\theta}^\infty
		e^{kre^{i\theta}} \frac{(re^{i\theta})^\alpha}{
			(re^{i\theta})^{1+\alpha} + \mu
		} e^{i\theta} \,\mathrm{d}r,
	\end{align*}
	and the fact $ \pi/2 < \theta < (\alpha+3)/(4\alpha+4)\pi $ implies
	\[
		\Snm{\frac{(re^{i\theta})^\alpha}{(re^{i\theta})^{1+\alpha}+\mu}} =
		\frac{r^\alpha}{
			\snm{
				r^{1+\alpha}\cos((1+\alpha)\theta) +
				\mu + i r^{1+\alpha}\sin((1+\alpha)\theta)
			}
		} < C_\alpha r^{-1}.
	\]
	Consequently,
	\begin{align*}
		\snm{\mathbb I_1}
		& \leqslant C_\alpha \snm{\xi_0} \int_{\pi/\sin\theta}^\infty
		e^{kr\cos\theta} r^{-1} \, \mathrm{d}r
		\leqslant C_\alpha \snm{\xi_0} \int_{\pi/\sin\theta}^\infty
		e^{kr\cos\theta} \, \mathrm{d}r \\
		& \leqslant C_\alpha k^{-1} e^{k\pi\cot\theta} \snm{\xi_0}.
		%\leqslant C_\alpha k^{-2} \snm{\xi_0}.
	\end{align*}

	Then let us estimate $ \mathbb I_2 $. For any $ z \in \Upsilon_1 \setminus \{0\} $,
	since
	\begin{align*}
		& z + \mu z^{-\alpha} = z^{-\alpha}(z^{1+\alpha} + \mu) \\
		={} &
		\snm{z}^{-\alpha} e^{-i\alpha\theta}
		\Big(
			\snm{z}^{1+\alpha} \cos\big((1+\alpha)\theta\big) + \mu +
			i \snm{z}^{1+\alpha} \sin\big((1+\alpha)\theta\big)
		\Big),
	\end{align*}
	from the fact $ \pi/2 < \theta < (\alpha+3)/(4\alpha+4)\pi $ it follows that
	\[
		%\label{eq:upsilon_1-1}
		\snm{z+\mu z^{-\alpha}} > C_\alpha \snm{z}.
	\]
	By \cref{eq:psi}, a routine calculation gives
	\[
		%\label{eq:upsilon_1-2}
		\snm{(1+\mu\psi(z))(e^z-1) - (z+\mu z^{-\alpha})}
		\leqslant C_\alpha\big( \snm{z}^2 + \mu \snm{z}^{1-\alpha} \big),
	\]
	and, similar to \cref{eq:psi-growth}, we have
	\[
		%\label{eq:upsilon_1-3}
		\snm{1+\mu\psi(z)} > C_\alpha (1+\mu \snm{z}^{-1-\alpha}).
	\]
	In addition, it is clear that
	\[
		\snm{e^z-1} > C_\alpha \snm{z}, \quad z \in \Upsilon_1 \setminus \{0\}.
	\]
	Using the above four estimates, we obtain
	\begin{align*}
		& \snm{
			(z+\mu z^{-\alpha})^{-1} - (1+\mu\psi(z))^{-1}(e^z-1)^{-1}
		} \\
		={} &
		\Snm{
			\frac{
				(1+\mu\psi(z))(e^z-1) - (z+\mu z^{-\alpha})
			}{
				(z+\mu z^{-\alpha})(1+\mu\psi(z))(e^z-1)
			}
		} \\
		<{} &
		C_\alpha \frac{
			\snm{z}^2 + \mu \snm{z}^{1-\alpha}
		}{
			\snm{z}^2(1+\mu\snm{z}^{-1-\alpha})
		} = C_\alpha
	\end{align*}
	for all $ z \in \Upsilon_1 \setminus \{0\} $. Therefore,
	\begin{small}
	\begin{align*}
		\snm{\mathbb I_2} &= \Snm{
			\frac{\xi_0}\pi \operatorname{Im} \int_0^{\pi\!/\!\sin\theta}
			\! e^{kre^{i\theta}} \! \Big(
				\big( re^{i\theta} \!+\! \mu (re^{i\theta})^{-\alpha} \big)^{-1}
				\!-\!  \big(1\!+\!\mu\psi(re^{i\theta})\big)^{-1}
				\big( e^{re^{i\theta}} \!-\! 1 \big)^{-1}
			\Big) e^{i\theta} \mathrm{d}r
		} \\
		& \leqslant C_\alpha \snm{\xi_0} \int_0^{\pi/\sin\theta}
		e^{kr\cos\theta} \, \mathrm{d}r
		\leqslant C_\alpha k^{-1} \snm{\xi_0}.
	\end{align*}
	\end{small}

	Finally, combing \cref{eq:y_k-Y_k} and the above estimates for $ \mathbb I_1 $ and
	$ \mathbb I_2 $ proves \cref{eq:y-Y} and thus concludes the proof.
\hfill\ensuremath{\blacksquare}

\subsection{The second fractional ordinary equation}
This subsection considers the fractional ordinary equation
\begin{equation}
	\xi'(t) + \lambda \D_{0+}^{-\alpha} \xi(t) = 1, \quad t > 0,
\end{equation}
subjected to the initial value condition $ \xi(0) = 0 $. Applying the temporal
discretization in \cref{eq:numer_sol} yields the following discretization: let $ Y_0
= 0 $; for $ k \in \mathbb N $, the value of $ Y_{k+1} $ is determined by that
\begin{equation}
	\mu \bigg(
		\sum_{j=1}^k Y_j(b_{k-j+2} - 2b_{k-j+1} + b_{k-j}) +
		b_1 Y_{k+1}
	\bigg)  + Y_{k+1} - Y_k = \tau.
\end{equation}
Similar to \cref{eq:y,eq:Y_k-3}, we have
\begin{align}
	\xi(t) &= \frac1{2\pi i}
	\int_{\Upsilon} e^{tz} (z^2+\lambda z^{1-\alpha})^{-1} \, \mathrm{d}z,
	\quad t > 0, \label{eq:y2} \\
	Y_k &= \frac\tau{2\pi i} \int_{\Upsilon_1}
	\frac{e^{kz+z}}{1+\mu\psi(z)} \frac{\mathrm{d}z}{(e^z-1)^2},
	\quad k \in \mathbb N_{>0}.
	\label{eq:Z_k}
\end{align}
%\begin{equation}
%\label{eq:y2}
%\xi(t) = \frac1{2\pi i}
%\int_{\Upsilon} e^{tz} z^{\alpha-1}(z^{1+\alpha}+\lambda)^{-1} \, \mathrm{d}z,
%\quad t > 0,
%\end{equation}
%\begin{equation}
%\label{eq:Z_k}
%Y_k = \frac\tau{2\pi i} \int_{\Upsilon_1}
%\frac{e^{kz+z}}{1+\mu\psi(z)} \frac{\mathrm{d}z}{(e^z-1)^2}.
%\end{equation}
%\begin{equation}
%\widehat Y(z) = \frac{
%\tau e^z
%}{
%1 + \mu\psi(z)
%}(e^z-1)^{-2}
%\end{equation}
%We have
%\begin{equation}
%Y_k = \frac\tau{2\pi i}
%\int_{a-i\pi}^{a+i\pi} \widehat Y(z) e^{kz} \, \mathrm{d}z =
%\frac\tau{2\pi i}
%\int_{a-i\pi}^{a+i\pi}
%\frac{e^{kz+z}}{1+\mu\psi(z)}
%\frac{\mathrm{d}z}{(e^z-1)^2}
%\end{equation}

\begin{thm}
	\label{thm:z-Z}
	For any $ k \in \mathbb N_{>0} $,
	\begin{equation}
		\label{eq:z-Z}
		\snm{\xi(t_k) - Y_k} < C_\alpha \tau.
	\end{equation}
\end{thm}
\begin{proof}
	Since the proof of this theorem is similar to that of \cref{thm:y-Y}, we only
	highlight the differences. Proceeding as in the proof of \cref{thm:y-Y} yields
	\[
		\xi(t_k) - Y_k = \mathbb I_1 + \mathbb I_2,
	\]
	where
	\begin{align*}
		\mathbb I_1 &:= \frac\tau{2\pi i}
		\int_{\Upsilon \setminus \Upsilon_1} e^{kz}
		(z^2+\mu z^{1-\alpha})^{-1} \, \mathrm{d}z, \\
		\mathbb I_2 &:=
		\frac\tau{2\pi i} \int_{\Upsilon_1} e^{kz} \Big(
			(z^2 + \mu z^{1-\alpha})^{-1} -
			(1+\mu\psi(z))^{-1} (e^z-1)^{-2} e^z
		\Big) \, \mathrm{d}z.
	\end{align*}
	Moreover,
	\begin{align*}
		\snm{\mathbb I_1} < C_\alpha \tau \int_{\pi/\sin\theta}^\infty
		e^{kr\cos\theta} r^{-2} \, \mathrm{d}r
		< C_\alpha \tau \int_{\pi/\sin\theta}^\infty
		e^{kr\cos\theta} \, \mathrm{d}r
		< C_\alpha \tau k^{-1} e^{k\pi\cot\theta}.
	\end{align*}
	For any $ z \in \Upsilon_1 \setminus \{0\} $, since
	\begin{align*}
		& z^2 + \mu z^{1-\alpha} = z^{1-\alpha}(z^{1+\alpha} + \mu) \\
		={} &
		\snm{z}^{1-\alpha} e^{i(1-\alpha)\theta}
		\Big(
			\snm{z}^{1+\alpha} \cos\big((1+\alpha)\theta\big) + \mu +
			i \snm{z}^{1+\alpha} \sin\big((1+\alpha)\theta\big)
		\Big),
	\end{align*}
	from the fact $ \pi/2 < \theta < (\alpha+3)/(4\alpha+4)\pi $ it follows that there
	exits a positive constant $ c $, depending only on $ \alpha $, such that
	\[
		\snm{z^2 + \mu z^{1-\alpha}} >
		\begin{cases}
			C_\alpha \mu \snm{z}^{1-\alpha} &
			\text{if}\quad 0 < \snm{z} \leqslant c\mu^{1/(1+\alpha)} , \\
			C_\alpha \snm{z}^2 &
			\text{if}\quad c \mu^{1/(1+\alpha)} \leqslant
			\snm{z} \leqslant \pi/\sin\theta.
		\end{cases}
	\]
	By \cref{eq:psi}, a routine calculation gives
	\[
		%\label{eq:upsilon_1-2}
		\snm{(1+\mu\psi(z))(e^z-1)^2 - (z^2+\mu z^{1-\alpha}) e^z}
		< C_\alpha \big( \snm{z}^4 + \mu\snm{z}^{2-\alpha} \big),
	\]
	and, similar to \cref{eq:psi-growth}, we have
	\[
		%\label{eq:upsilon_1-3}
		\snm{1+\mu\psi(z)} > C_\alpha (1+\mu \snm{z}^{-1-\alpha}).
	\]
	Using the above three estimates, we obtain
	\begin{align*}
		& \snm{
			(z^2+\mu z^{1-\alpha})^{-1} - (1+\mu\psi(z))^{-1}(e^z-1)^{-2}e^z
		} \\
		={} &
		\Snm{
			\frac{
				(1+\mu\psi(z))(e^z-1)^2 - (z^2+\mu z^{1-\alpha})e^z
			}{
				(z^2+\mu z^{1-\alpha})(1+\mu\psi(z))(e^z-1)^2
			}
		} \\
		<{} &
		\left\{
			\begin{aligned}
				& C_\alpha \frac{\snm{z}^4 + \mu\snm{z}^{2-\alpha}}
				{\mu(\snm{z}^{3-\alpha} + \mu\snm{z}^{2-2\alpha})}
				\qquad\quad\text{ if }\,\, 0 < \snm{z} \leqslant c\mu^{1/(1+\alpha)}, \\
				& C_\alpha \frac{\snm{z}^4 + \mu\snm{z}^{2-\alpha}}
				{\snm{z}^4 + \mu\snm{z}^{3-\alpha}}
				\qquad\qquad\qquad\text{ if}\,\, c\mu^{1/(1+\alpha)} < \snm{z}
				\leqslant \pi/\sin\theta,
			\end{aligned}
		\right. \\
		<{} &
		\left\{
			\begin{aligned}
				& C_\alpha \big(
					1 + \mu^{-1} \snm{z}^\alpha
				\big) \qquad\qquad\qquad\,\,\text{ if }
				\,\, 0 < \snm{z} \leqslant c\mu^{1/(1+\alpha)}, \\
				& C_\alpha \big( 1 + \mu \snm{z}^{-2-\alpha} \big)
				\qquad\qquad\qquad\text{ if}
				\,\, c\mu^{1/(1+\alpha)} < \snm{z} < \pi/\sin\theta,
			\end{aligned}
		\right.
	\end{align*}
	for all $ z \in \Upsilon_1 \setminus \{0\} $. Therefore, if $ c\mu^{1/(1+\alpha)}
	\leqslant \pi/\sin\theta $ then
	\begin{align*}
		\snm{\mathbb I_2} & < C_\alpha \tau \bigg(
			\int_0^{c\mu^{1/(1+\alpha)}}
			e^{kr\cos\theta} \big(
				1 + \mu^{-1} r^\alpha
			\big) \, \mathrm{d}r \\
			& \qquad\qquad {} +
			\int_{ c\mu^{1/(1+\alpha)} }^{ \pi/\sin\theta }
			e^{kr\cos\theta}(1 + \mu r^{-2-\alpha}) \, \mathrm{d}r
		\bigg) \\
		& < C_\alpha \tau \Big(
			\int_0^{c\mu^{1/(1+\alpha)}}
			1 + \mu^{-1} r^\alpha \, \mathrm{d} r +
			\int_{c\mu^{1/(1+\alpha)}}^{\pi/\sin\theta}
			1 + \mu r^{-2-\alpha} \, \mathrm{d}r
		\Big) \\
		& < C_\alpha \tau,
	\end{align*}
	and if $ c\mu^{1/(1+\alpha)} > \pi/\sin\theta $ then
	\[
		\snm{\mathbb I_2}  < C_\alpha \tau
		\int_0^{\pi\sin\theta}
		e^{kr\cos\theta} \big(
			1 + \mu^{-1} r^\alpha
		\big) \, \mathrm{d}r < C_\alpha\tau.
	\]

	Finally, combing the above estimates for $ \mathbb I_1 $ and $ \mathbb I_2 $ proves
	\cref{eq:z-Z} and hence this theorem.
\end{proof}

\section{Main results}
\label{sec:main}
In the rest of this paper, we assume that $ h < e^{-2(1+\alpha)} $ and $ \tau < T/e
$. The symbol $ a\lesssim b $ means that there exists a positive constant $ C $,
depending only on $ \alpha $, $ T $, $ \Omega $, the shape-regular parameter of $
\mathcal K_h $ and the ratio of $ h $ to the minimum diameter of the elements in $
\mathcal K_h $, unless otherwise specified, such that $ a \leqslant Cb $.
Additionally, since the following properties are frequently used in the forthcoming
analysis, we will use them implicitly (cf.~\cite{Samko1993}):
\begin{align*}
	& \D_{a+}^\beta \D_{a+}^\gamma =
	\D_{a+}^{\beta+\gamma}, \quad
	\D_{b-}^\beta \D_{b-}^\gamma = \D_{b-}^{\beta+\gamma},
	\text{ and } \\
	& \dual{\D_{a+}^\beta v, w}_{(a,b)} =
	\dual{v, \D_{b-}^\beta w}_{(a,b)},
	\quad v,w \in L^2(a,b),
\end{align*}
where $ -\infty < a < b < \infty $ and $ -\infty < \beta, \gamma < 0 $.

\begin{thm}
	\label{thm:conv-u0}
	If $ u_0 \in L^2(\Omega) $ and $ f = 0 $, then
	\begin{equation}
		\label{eq:conv-u0}
		\nm{u(t_j) - U_j}_{L^2(\Omega)} \lesssim
		\big(
			h^2 t_j^{-\alpha-1} + \tau t_j^{-1}
		\big) \nm{u_0}_{L^2(\Omega)}
	\end{equation}
	for all $ 1 \leqslant j \leqslant J $.
\end{thm}
\begin{proof}
	Let $ u_h $ be the solution of the spatially discrete problem:
	\[
		u_h'(t) - \Delta_h \D_{0+}^{-\alpha} u_h(t) = 0,
		\quad t > 0,
	\]
	subjected to the initial value condition $ u_h(0) = P_h u_0 $, where the discrete
	Laplace operator $ \Delta_h:S_h \to S_h $ is defined by that
	\[
		\dual{-\Delta_h v_h,w_h}_\Omega :=
		\dual{\nabla v_h, \nabla w_h}_\Omega
		\quad\text{for all}\, v_h, w_h \in S_h.
	\]
	By \cite[Theorem
	2.1]{Lubich1996} we have
	\[
		\nm{u(t) - u_h(t)}_{L^2(\Omega)} \lesssim
		h^2 t^{-\alpha-1} \nm{u_0}_{L^2(\Omega)},
		\quad t > 0,
	\]
	and by \cref{thm:y-Y} we obtain
	\[
		\nm{U_j - u_h(t_j)}_{L^2(\Omega)}
		\lesssim \tau t_j^{-1} \nm{u_0}_{L^2(\Omega)}.
	\]
	Combining the above two estimates proves \cref{eq:conv-u0} and hence this theorem.
\end{proof}

\begin{thm}
	\label{thm:f-const}
	If $ u_0 = 0 $ and $ f(t) = v \in L^2(\Omega) $, $ 0 < t < T $, then
	\begin{equation}
		\label{eq:f-const}
		\nm{u(t_j) - U_j}_{L^2(\Omega)} \lesssim
		\big( t_j^{-\alpha} h^2 + \tau \big) \nm{v}_{L^2(\Omega)}
	\end{equation}
	for all $ 1 \leqslant j \leqslant J $.
\end{thm}
\begin{proof}
	Let $ u_h $ be the solution of the spatially discrete problem:
	\[
		u_h'(t) - \Delta_h \D_{0+}^{-\alpha} u_h(t) = P_h v,
		\quad t > 0,
	\]
	subjected to the initial value condition $ u_h(0) = 0 $. By \cite[Theorem
	2.2]{Lubich1996} we have
	\[
		\nm{u(t) - u_h(t)}_{L^2(\Omega)} \lesssim
		t^{-\alpha} h^2 \nm{v}_{L^2(\Omega)},
		\quad t > 0,
	\]
	and \cref{thm:z-Z} implies
	\[
		\nm{U_j - u_h(t_j)}_{L^2(\Omega)} \lesssim
		\tau \nm{P_hv}_{L^2(\Omega)} \lesssim
		\tau \nm{v}_{L^2(\Omega)}.
	\]
	Combining the above two estimates proves \cref{eq:f-const} and hence this theorem.
\end{proof}

\begin{thm}
	\label{thm:conv_f_L2}
	If $ u_0 = 0 $ and $ f \in L^2(0,T;\dot H^{\alpha/(\alpha+1)}(\Omega)\!) $, then
	\begin{equation}
		\label{eq:conv_f_L2}
		\nm{u-U}_{L^\infty(0,T;L^2(\Omega))}
		\lesssim \left( h + \sqrt{\ln(1/h)} \tau^{1/2} \right)
		\nm{f}_{L^2(0,T;\dot H^{\alpha/(\alpha+1)}(\Omega)\!)}.
	\end{equation}
\end{thm}
\begin{rem}
	Since \cref{thm:regu-pde} implies
	\begin{align*}
		\nm{u}_{{}_0H^1(0,T;\dot H^{\alpha/(\alpha+1)}(\Omega)\!)} +
		\nm{u}_{C([0,T];\dot H^1(\Omega)\!)} & \leqslant
		C_\alpha \nm{f}_{L^2(0,T;\dot H^{\alpha/(\alpha+1)}(\Omega)\!)},
	\end{align*}
	error estimate \cref{eq:conv_f_L2} is nearly optimal with respect to the
	regularity of $ u $.
	%In fact, a simple modification of the proof of \cref{thm:conv_f_L2}
	%yields another error estimate
	%\[
	%\nm{u-U}_{L^\infty(0,T;L^2(\Omega)\!)} \lesssim
	%\left( h + h^{-1} \tau^{\alpha+1} + \tau^{(\alpha+1)/2} \right)
	%\nm{f}_{{}_0H^{\alpha/2}(0,T;L^2(\Omega)\!)}.
	%\]
	%According to the optimal convergence rate $ O(h + \tau^{(\alpha+1)/2}) $, in
	%practical computation $ h $ and $ \tau $ should satisfy that $ h =
	%\tau^{(\alpha+1)/2} $, and in this case the above error estimate is obviously
	%optimal.
\end{rem}
\begin{thm}
	\label{thm:conv_f_higher}
	If $ u_0 = 0 $ and $ f \in {}_0H^{\alpha+1/2}(0,T;L^2(\Omega)) $, then
	\begin{equation}
		\label{eq:conv_f_higher}
		\nm{u\!-\!U}_{L^\infty(0,T;L^2(\Omega))} \lesssim
		\ln(T/\tau) \left( \sqrt{\ln(1/h)}\, h^2 \!+\! \tau \right)
		\nm{f}_{{}_0H^{\alpha+1/2}(0,T;L^2(\Omega))}.
	\end{equation}
\end{thm}

\begin{rem}
	Assume that $ u $ satisfies the following regularity assumption: for any $ 0 < t
	\leqslant T $,
	\begin{align*}
		\nm{u(t)}_{\dot H^2(\Omega)} + t\nm{u'(t)}_{\dot H^2(\Omega)}
		& \leqslant M, \\
		\nm{u'(t)}_{L^2(\Omega)} + t\nm{u''(t)}_{L^2(\Omega)}
		& \leqslant M t^{\sigma-1}, \\
		t\nm{u'(t)}_{\dot H^2(\Omega)} +
		t^2 \nm{u''(t)}_{\dot H^2(\Omega)} & \leqslant Mt^{\sigma-1},
	\end{align*}
	where $ M $ and $ \sigma $ are two positive constants. Letting
	\[
		t_j = (j/J)^\gamma T \text{ for all } 1 \leqslant j \leqslant J,
		\quad \gamma > 1/\sigma,
	\]
	Mustapha and McLean \cite{Mustapha2009Discontinuous} obtained
	\[
		\nm{u(t_j)-U_j}_{L^2(\Omega)}
		\lesssim \nm{u_0 - U_0}_{L^2(\Omega)} + M \big( \ln(t_j/t_1)h^2 + T/J \big),
	\]
	and hence in the case $ u_0 \in L^2(\Omega) $ no convergence rate was derived.
	Besides, under the condition $ u_0 = 0 $ and $ f \in
	{}_0H^{\alpha+1/2}(0,T;L^2(\Omega)) $, by \cref{thm:regu-pde} we have only
	\[
		\nm{u}_{{}_0H^{\alpha+3/2}(0,T;L^2(\Omega))} +
		\nm{u}_{{}_0H^{1/2}(0,T;\dot H^2(\Omega))}
		\leqslant C_\alpha \nm{f}_{{}_0H^{\alpha+1/2}(0,T;L^2(\Omega))},
	\]
	so that $ u $ does not satisfy the above regularity assumption necessarily.
\end{rem}

\begin{rem}
	\cref{thm:f-const,thm:conv_f_higher} imply that if $ f \in
	H^{\alpha+1/2}(0,T;L^2(\Omega)) $ and $ f(0) \neq 0 $, then
	\begin{small}
	\[
		\nm{u(t_j) - U_j}_{L^2(\Omega)} \lesssim
		\left(
			\left( \ln(T/\tau) \sqrt{\ln(1/h)} + t_j^{-\alpha} \right) h^2 +
			\ln(T/\tau) \tau
		\right) \nm{f}_{H^{\alpha+1/2}(0,T;L^2(\Omega))}
	\]
	\end{small}
	for all $ 1 \leqslant j \leqslant J $, where $ H^{\alpha+1/2}(0,T;L^2(\Omega)) $ is
	defined analogously to the space $ {}_0H^{\alpha+1/2}(0,T;L^2(\Omega)) $.
	Furthermore, \cref{thm:conv-u0,thm:f-const} imply that if the accuracy of $ U $
	near $ t = 0 $ is unimportant, then using graded temporal grids to tackle the
	singularity caused by nonsmooth $ u_0 $ and $ f(0) $ is unnecessary.
\end{rem}

The rest of this section is devoted to the proofs of \cref{thm:conv_f_L2,thm:conv_f_higher}.
Let $ X $ be a separable Hilbert space. For any $ w \in C((0,T];X) $ we define
\[
	(P_\tau v)|_{I_j} \equiv v(t_j),
	\quad 1 \leqslant j \leqslant J,
\]
and for any $ v \in L^1(0,T;X) $ we define
\begin{equation}
	\label{eq:def-Q}
	(Q_\tau v)|_{I_j} \equiv \tau^{-1} \int_{I_j} v,
	\quad 1 \leqslant j \leqslant J.
\end{equation}
The operator $ Q_\tau $ possesses the standard estimates
\begin{align*}
	\nm{(I-Q_\tau)v}_{L^2(0,T;X)} &\leqslant \nm{v}_{L^2(0,T;X)}
	\phantom{\tau{}_0}\quad\forall v \in L^2(0,T;X), \\
	\nm{(I-Q_\tau)v}_{L^2(0,T;X)} &\lesssim \tau \nm{v}_{{}_0H^1(0,T;X)}
	\quad\forall v \in {}_0H^1(0,T;X).
\end{align*}
Hence, for any $ v \in {}_0H^\beta(0,T;X) $ with $ 0 < \beta < 1 $, applying
\cite[Lemma~22.3]{Tartar2007} yields
\[
	\nm{(I-Q_\tau)v}_{[L^2(0,T;X),\ L^2(0,T;X)]_{\beta,2}}
	\lesssim \tau^\beta \nm{v}_{{}_0H^\beta(0,T;X)},
\]
so that \cite[(23.11)]{Tartar2007} implies
\begin{equation}
	\label{eq:Q_tau}
	\nm{(I-Q_\tau)v}_{L^2(0,T;X)} \lesssim
	\tau^\beta \sqrt{\beta(1-\beta)} \, \nm{v}_{{}_0H^\beta(0,T;X)}.
\end{equation}
Here we have used the fact that $ {}_0H^\beta(0,T;X) = [L^2(0,T;X),
{}_0H^1(0,T;X)]_{\beta,2} $ with equivalent norms (cf.~\cref{rem:equiv_frac_space}).
Similarly, for any $ v \in {}^0H^\beta(0,T;X) $ with $ 0 < \beta < 1 $.
\begin{equation}
	\label{eq:Q_tau-sys}
	\nm{(I-Q_\tau)v}_{L^2(0,T;X)} \lesssim
	\tau^\beta \sqrt{\beta(1-\beta)}
	\, \nm{v}_{{}^0H^\beta(0,T;X)}.
\end{equation}
Moreover, from \cite[Lemmas 12.4, 16.3, 22.3, 23.1]{Tartar2007} it follows the
following three well-known estimates.
\begin{lem}
	\label{lem:Hs}
	If $ v \in {}_0H^\beta(0,1) $ with $ 0 < \beta < 1 $, then
	\begin{equation}
		\label{eq:Hs-1}
		\left(
			\int_0^1 \int_0^1 \frac{\snm{v(t)-v(s)}^2}{\snm{t-s}^{1+2\beta}}
			\, \mathrm{d}t \, \mathrm{d}s
		\right)^{1/2} \leqslant C \nm{v}_{{}_0H^\beta(0,1)},
	\end{equation}
	and if, in addition, $ 1/2 < \beta < 1 $ then
	\begin{equation}
		\label{eq:Hs-2}
		\nm{v}_{C[0,1]} \leqslant C
		\sqrt{\frac{1-\beta}{2\beta-1}} \nm{v}_{{}_0H^\beta(0,1)},
	\end{equation}
	\begin{small}
	\begin{equation}
		\label{eq:Hs-3}
		\nm{v}_{C[0,1]} \leqslant
		\frac{C}{\sqrt{2\beta\!-\!1}} \left(
			\nm{v}_{L^2(0,1)} \!+\! \sqrt{1\!-\!\beta} \left(
				\int_0^1\!\!\int_0^1 \frac{\snm{v(t)\!-\!v(s)}^2}{\snm{t\!-\!s}^{1+2\beta}}
				\, \mathrm{d}s \, \mathrm{d}t
			\right)^{1/2}
		\right),
	\end{equation}
	\end{small}
	where $ C $ is a positive constant independent of $ \beta $ and $ v $.
\end{lem}

\begin{lem}
	\label{lem:P_tau}
	If $ v \in {}_0H^\beta(0,T) $ with $ 1/2 < \beta < 1 $, then
	\begin{equation}
		\label{eq:P_tau}
		\nm{(I-P_\tau)v}_{L^2(0,T)} \lesssim
		\tau^\beta \sqrt{\frac{1-\beta}{2\beta-1}}\,
		\nm{v}_{{}_0H^\beta(0,T)}.
	\end{equation}
\end{lem}
\begin{proof}
	By the definition of $ P_\tau $ and \cref{eq:Hs-3}, a scaling argument yields
	\begin{small}
	\begin{align*}
		& \nm{(I-P_\tau)v}_{L^2(I_j)}^2 \\
		\lesssim{} &
		\frac1{2\beta-1} \left(
			\nm{(I-Q_\tau)v}_{L^2(I_j)}^2 +
			(1-\beta)\tau^{2\beta} \int_{I_j} \int_{I_j}
			\frac{\snm{v(t)-v(s)}^2}{\snm{t-s}^{1+2\beta}}
			\, \mathrm{d}s \, \mathrm{d}t
		\right),
	\end{align*}
	\end{small}
	so that
	\begin{small}
	\begin{align*}
		& \sqrt{2\beta-1}
		\nm{(I-P_\tau)v}_{L^2(0,T)} \\
		\lesssim{} &
		\nm{(I-Q_\tau)v}_{L^2(0,T)} +
		\sqrt{1-\beta}\, \tau^\beta \left(
			\int_0^T \int_0^T \frac{\snm{v(t)-v(s)}^2}{\snm{t-s}^{1+2\beta}}
			\, \mathrm{d}s \, \mathrm{d}t
		\right)^{1/2} \\
		\lesssim{} &
		\tau^\beta \sqrt{1-\beta}\left(
			\nm{v}_{{}_0H^\beta(0,T)} + \left(
				\int_0^T \int_0^T \frac{\snm{v(t)-v(s)}^2}{\snm{t-s}^{1+2\beta}}
				\, \mathrm{d}s \, \mathrm{d}t
			\right)^{1/2}
		\right),
	\end{align*}
	\end{small}
	by \cref{eq:Q_tau}. Another scaling argument gives, by \cref{eq:Hs-1}, that
	\begin{small}
	\[
		\left(
			\int_0^T \int_0^T \frac{\snm{v(t)-v(s)}^2}{\snm{t-s}^{1+2\beta}}
			\, \mathrm{d}s \, \mathrm{d}t
		\right)^{1/2} \lesssim \nm{v}_{{}_0H^\beta(0,T)}.
	\]
	\end{small}
	Combining the above two estimates proves \cref{eq:P_tau} and thus concludes this
	proof.
\end{proof}

% Todo: cite a reference
\begin{lem}[\cite{Luo2018Convergence}]
	\label{lem:interp}
	Assume that $ -\infty < \beta,\gamma,r,s < \infty $ and $ 0 < \theta < 1 $. If $ v
	\in {}_0H^\beta(0,T;\dot H^r(\Omega)) \cap {}_0H^\gamma(0,T;\dot H^s(\Omega)) $,
	then
	\begin{equation}
		\label{eq:interp-1}
		\begin{aligned}
			& \nm{v}_{
				{}_0H^{(1-\theta)\beta + \theta \gamma}
				(0,T;\dot H^{(1-\theta)r+\theta s}(\Omega))
			} \\
			\leqslant{} &
			C_{\beta,\gamma,\theta}
			\nm{v}_{{}_0H^\beta(0,T;\dot H^r(\Omega)))}^{1-\theta}
			\nm{v}_{{}_0H^\gamma(0,T;\dot H^s(\Omega))}^\theta.
		\end{aligned}
	\end{equation}
	In particular, if $ \beta = 0 $ and $ \gamma = 1 $ then
	\begin{equation}
		\nm{v}_{{}_0H^\theta(0,T;\dot H^{(1-\theta)r + \theta s}(\Omega)}
		\leqslant \frac1{\sqrt{2\theta(1-\theta)}}
		\nm{v}_{L^2(0,T;\dot H^r(\Omega)}^{1-\theta}
		\nm{v}_{{}_0H^1(0,T;\dot H^s(\Omega))}^\theta
	\end{equation}
	for all $ v \in L^2(0,T;\dot H^r(\Omega)) \cap {}_0H^1(0,T;\dot H^s(\Omega)) $.
	%In particular, if $ 0 \leqslant \beta < \gamma \leqslant 1 $, then
	%\begin{equation}
	%\begin{aligned}
	%& \nm{v}_{
	%{}_0H^{(1-\theta)\beta + \theta \gamma}
	%(0,T;\dot H^{(1-\theta)r + \theta s}(\Omega))
	%} \\
	%\leqslant{} & \frac1{\sqrt{\theta(1-\theta)}}
	%\sqrt{
	%\frac\beta{\gamma-\beta}
	%\Big(\frac\gamma\beta\Big)^\theta
	%}
	%\nm{v}_{{}_0H^\beta(0,T;\dot H^r(\Omega))}^{1-\theta}
	%\nm{v}_{{}_0H^\gamma(0,T;\dot H^s(\Omega))}^\theta.
	%\end{aligned}
	%\end{equation}
\end{lem}

\begin{lem}[\cite{Thomee2006}]
	\label{lem:VV'}
	If $ V \in W_{\tau,h} $ and $ 0 \leqslant i < k \leqslant J $, then
	\[
		\sum_{j=i}^k \dual{\jmp{V_j}, V_j^{+}}_\Omega \geqslant
		\frac12 \big(
			\nm{V_k^{+}}_{L^2(\Omega)}^2 - \nm{V_i}_{L^2(\Omega)}^2
		\big) \geqslant
		\sum_{j=i}^k \dual{V_j, \jmp{V_j}}_\Omega.
	\]
\end{lem}

\subsection{Proof of \cref{thm:conv_f_L2}}
%\medskip\noindent{\bf Proof of \cref{thm:conv_f_L2}.}
Let us first prove
\begin{equation}
	\label{eq:inf-theta}
	\begin{aligned}
		& \nm{U-P_\tau P_hu}_{L^\infty(0,T;L^2(\Omega))} \\
		\lesssim{} &
		\nm{(I-P_h)\D_{0+}^{-\alpha/2} u}_{L^2(0,T;\dot H^{1}(\Omega)\!)} +
		\nm{(I-P_\tau)P_hu}_{L^2(0,T;\dot H^1(\Omega)\!)}.
	\end{aligned}
\end{equation}
For any $ 1 \leqslant j \leqslant J $, by \cref{eq:numer_sol,eq:strong-form} we have
\[
	\sum_{i=1}^{j} \dual{u',\theta}_{\Omega\times I_i} +
	\dual{
		\nabla \D_{0+}^{-\alpha} (u-U),\nabla \theta
	}_{\Omega \times (0,t_j)} =
	\sum_{i=0}^{j-1}\dual{\jmp{U_i},\theta^+_i}_\Omega,
\]
where $ \theta:=U-P_\tau P_hu $ and we set $ (P_\tau P_h u)_0 = 0 $. By the
definitions of $ P_h $ and $ P_\tau $, a routine calculation (see \cite[Chapter
12]{Thomee2006}) then yields
\begin{align*}
	{}&
	\sum_{i=0}^{j-1}\dual{\jmp{\theta_i},\theta^+_i}_\Omega+
	\big\langle
	\nabla \D_{0+}^{-\alpha} \theta,\nabla \theta
	\big\rangle_{\Omega \times (0,t_j)} \\
	={} &
	\dual{
		\nabla\D_{0+}^{-\alpha}(u-P_\tau P_hu), \nabla\theta
	}_{\Omega \times (0,t_j)} \\
	={} &
	\dual{
		\nabla\D_{0+}^{-\alpha}(I-P_h)u, \nabla\theta
	}_{\Omega \times (0,t_j)} +
	\dual{
		\nabla \D_{0+}^{-\alpha}(I-P_\tau)P_hu, \nabla\theta
	}_{\Omega \times (0,t_j)} \\
	={} &
	\dual{
		\nabla(I-P_h)\D_{0+}^{-\alpha/2}u,
		\nabla \D_{t_j-}^{-\alpha/2}\theta
	}_{\Omega \times (0,t_j)} +
	\dual{
		\nabla (I-P_\tau)P_hu, \D_{t_j-}^{-\alpha} \nabla\theta
	}_{\Omega \times (0,t_j)},
\end{align*}
so that using \cref{lem:VV'}, \cref{lem:coer}, Sobolev inequality and the Young's
inequality with $ \epsilon $ gives
\begin{align*}
	{}& \nm{\theta_j}_{L^2(\Omega)}+
	\nm{\theta_1}_{L^2(\Omega)}+
	\nm{ \D_{0+}^{-\alpha/2} \theta }_{L^{2}(0,t_j;\dot H^1(\Omega)\!)} \\
	\lesssim {}&
	\nm{(I-P_h)\D_{0+}^{-\alpha/2} u}_{L^2(0,T;\dot H^{1}(\Omega)\!)} +
	\nm{(I-P_\tau)P_hu}_{L^2(0,T;\dot H^1(\Omega)\!)}.
\end{align*}
Since $ 1 \leqslant j \leqslant J $ is arbitrary, this implies \cref{eq:inf-theta}.

Next, let us prove
\begin{equation}
	\label{eq:inf-theta-2}
	\nm{U - P_\tau P_hu}_{L^\infty(0,T;L^2(\Omega))}
	\lesssim \big( h + \sqrt{\ln(1/h)} \, \tau^{1/2} \big)
	\nm{f}_{ L^2(0,T;\dot H^{\alpha/(\alpha+1)}(\Omega)) }.
\end{equation}
By the inverse estimate and \cref{lem:P_tau}, a straightforward computation gives
that, for any $ 0 < \epsilon < 1/(\alpha+1) $,
\begin{align*}
	& \nm{(I-P_\tau)P_hu}_{L^2(0,T;\dot H^1(\Omega)\!)} \lesssim
	h^{-\epsilon} \nm{(I-P_\tau)P_hu}_{L^2(0,T;\dot H^{1-\epsilon}(\Omega)\!)} \\
	\lesssim{} &
	h^{-\epsilon} \nm{(I-P_\tau)u}_{L^2(0,T;\dot H^{1-\epsilon}(\Omega))} \\
	\lesssim{} &
	h^{-\epsilon} \tau^{(1+\epsilon+\epsilon\alpha)/2}
	\sqrt{ \frac{1-(1+\alpha)\epsilon}\epsilon } \,
	\nm{u}_{
		{}_0H^{(1+\epsilon+\epsilon\alpha)/2}
		(0,T;\dot H^{1-\epsilon}(\Omega))
	},
\end{align*}
and hence letting $ \epsilon = (2\ln(1/h)\!)^{-1} $ yields
\[
	\nm{(I-P_\tau)P_hu}_{L^2(0,T;\dot H^1(\Omega)\!)}
	\lesssim \sqrt{\ln(1/h)} \, \tau^{1/2}
	\nm{u}_{
		{}_0H^{(1+\epsilon+\epsilon\alpha)/2}
		(0,T;\dot H^{1-\epsilon}(\Omega))
	}.
\]
Moreover, by \cref{lem:regu} we have
\begin{align*}
	\nm{(I-P_h)\D_{0+}^{-\alpha/2} u}_{L^2(0,T;\dot H^1(\Omega)\!)}
	& \lesssim h \nm{\D_{0+}^{-\alpha/2} u}_{L^2(0,T;\dot H^2(\Omega)\!)} \\
	& \lesssim
	h \nm{u}_{{}_0H^{-\alpha/2}(0,T;\dot H^2(\Omega)\!)}.
\end{align*}
Therefore, by \cref{thm:regu-pde,lem:interp}, combining \cref{eq:inf-theta} and the
above two estimates yields \cref{eq:inf-theta-2}.

Finally, a routine calculation gives
\begin{align*}
	& \nm{u-P_\tau P_hu}_{L^\infty(0,T;L^2(\Omega))} \\
	\leqslant{} &
	\nm{(I-P_h)u}_{L^\infty(0,T;L^2(\Omega))} +
	\nm{P_h(I-P_\tau)u}_{L^\infty(0,T;L^2(\Omega))} \\
	\leqslant{} &
	\nm{(I-P_h)u}_{L^\infty(0,T;L^2(\Omega))} +
	\nm{(I-P_\tau)u}_{L^\infty(0,T;L^2(\Omega))} \\
	\lesssim{} &
	h \nm{u}_{C([0,T];\dot H^1(\Omega))} +
	\tau^{1/2} \nm{u}_{H^1(0,T;L^2(\Omega))} \\
	\lesssim{} &
	\big( h + \tau^{1/2} \big)
	\nm{f}_{L^2(0,T;\dot H^{\alpha/(\alpha+1)}(\Omega))}
	\quad\text{(by \cref{thm:regu-pde})},
\end{align*}
so that \cref{eq:conv_f_L2} follows from \cref{eq:inf-theta-2} and the triangle
inequality
\begin{small}
\[
	\nm{u-U}_{L^\infty(0,T;L^2(\Omega))} \leqslant
	\nm{U-P_\tau P_hu}_{L^\infty(0,T;L^2(\Omega))} +
	\nm{u-P_\tau P_hu}_{L^\infty(0,T;L^2(\Omega))}.
\]
\end{small}
This completes the proof of \cref{thm:conv_f_L2}.
\begin{rem}
	From the above proof, it is easy to see that \cref{thm:conv_f_L2} still holds for the case of variable time steps.
\end{rem}

\subsection{Proof of \cref{thm:conv_f_higher}}
\begin{lem}
	\label{thm:discrete_regu}
	If $ W \in W_{\tau,h} $ satisfies that $ W_0 := v_h \in S_h $ and
	\begin{equation}
		\label{eq:W}
		\sum_{j=0}^{J-1} \dual{\jmp{W_j}, V_j^{+}}_\Omega +
		\dual{\nabla\D_{0+}^{-\alpha} W, \nabla V}_{\Omega \times (0,T)} = 0
		\quad \forall V \in W_{\tau,h},
	\end{equation}
	then
	\begin{align}
		\nm{W}_{{}_0H^{-\alpha/2}(0,T;\dot H^1(\Omega))} &
		\leqslant C_\alpha \nm{v_h}_{L^2(\Omega)},
		\label{eq:disc_regu_1} \\
		\nm{
			Q_\tau\D_{0+}^{-\alpha}(-\Delta_h W)
		}_{L^1(0,T;L^2(\Omega))} &
		\leqslant C_\alpha \ln(T/\tau) \nm{v_h}_{L^2(\Omega)}.
		\label{eq:disc_regu_2}
	\end{align}
\end{lem}
\begin{proof}
	Since \cref{lem:VV'} implies
	\[
		\sum_{j=0}^{J-1} \dual{\jmp{W_j}, W_j^{+}}_\Omega \geqslant
		\frac12 \big(
			\nm{W_J}_{L^2(\Omega)} - \nm{v_h}_{L^2(\Omega)}^2
		\big),
	\]
	inserting $ V = W $ into \cref{eq:W} yields
	\[
		\frac12 \nm{W_J}_{L^2(\Omega)}^2 +
		\dual{\nabla\D_{0+}^{-\alpha} W, \nabla W}_{\Omega \times (0,T)}
		\leqslant \frac12 \nm{v_h}_{L^2(\Omega)}^2.
	\]
	Hence, using \cref{lem:coer,lem:regu} proves \cref{eq:disc_regu_1}.

	Now let us prove \cref{eq:disc_regu_2}. Let $ \{\phi_{n,h}: 1 \leqslant n \leqslant
	N\} $ be an orthonormal basis of $ S_h $ endowed with the norm $ L^2(\Omega) $ such
	that
	\[
		-\Delta_h \phi_{n,h} = \lambda_{n,h} \phi_{n,h},
	\]
	where $ \{\lambda_{n,h}: 1 \leqslant n \leqslant N \} $ is the set of all
	eigenvalues of $ -\Delta_h $. For each $ 1 \leqslant n \leqslant N $, define $
	(Y_k^n)_{k=0}^\infty $ as that described in the first paragraph of
	\cref{ssec:first_ode} with $ \xi_0 $ replaced by $ \dual{v_h,\phi_{n,h}}_\Omega $
	and $ \lambda $ replaced by $ \lambda_{n,h} $. We also define $ W^n(t) :=
	\dual{W(t), \phi_{n,h}}_\Omega $, $ 0 < t < T $, and it is easy to verify that
	\[
		W^n = Y_j^n \quad \text{ on } I_j,\quad 1 \leqslant j \leqslant J.
	\]
	Hence, \cref{thm:Y-jump} implies
	\[
		\nm{\jmp{W_j}}_{L^2(\Omega)} \leqslant
		C_\alpha j^{-1} \nm{v_h}_{L^2(\Omega)},\quad 1 \leqslant j < J,
	\]
	and then it follows that
	\begin{equation}
		\label{eq:disc_regu_3}
		\sum_{j=1}^{J-1} \nm{\jmp{W_j}}_{L^2(\Omega)}
		\leqslant C_\alpha \nm{v_h}_{L^2(\Omega)}
		\sum_{j=1}^{J-1} j^{-1}
		\leqslant C_\alpha \ln(T/\tau) \nm{v_h}_{L^2(\Omega)}.
	\end{equation}
	In addition, inserting $ V = W \chi_{(0,t_1)} $ into \cref{eq:W} yields, by
	\cref{lem:coer}, that
	\[
		\nm{W_1}_{L^2(\Omega)} \leqslant \nm{W_0}_{L^2(\Omega)},
	\]
	which implies
	\begin{equation}
		\label{eq:disc_regu_4}
		\nm{\jmp{W_0}}_{L^2(\Omega)} \leqslant 2 \nm{W_0}_{L^2(\Omega)} =
		2\nm{v_h}_{L^2(\Omega)}.
	\end{equation}
	Consequently, since \cref{eq:W} implies
	\[
		\tau Q_\tau \D_{0+}^{-\alpha}(-\Delta_h W) =
		\jmp{W_{j-1}} \quad \text{ on } I_j,
		\quad 1 \leqslant j \leqslant J,
	\]
	combining \cref{eq:disc_regu_3,eq:disc_regu_4} proves \cref{eq:disc_regu_2} and
	hence this lemma.
\end{proof}

%Symmetrically, we have the following discrete regularity estimates.
%\begin{thm}
%\label{thm:discrete_regu_dual}
%If $ W \in W_{\tau,h} $ satisfies that $ W_J^{+} := v_h \in S_h $ and
%%\begin{small}
%\[
%-\sum_{j=1}^J \dual{V_j, \jmp{W_j}}_\Omega +
%\dual{\nabla V, \nabla\D_{T-}^{-\alpha} W}_{\Omega \times (0,T)} = 0
%\quad \forall V \in W_{\tau,h},
%\]
%%\end{small}
%then
%\begin{align*}
%\nm{W}_{{}^0H^{-\alpha/2}(0,T;\dot H^1(\Omega))} &
%\leqslant C_\alpha \nm{v_h}_{L^2(\Omega)}, \\
%\nm{
%Q_\tau\D_{T-}^{-\alpha}(-\Delta_h W)
%}_{L^1(0,T;L^2(\Omega))} &
%\leqslant C_\alpha \big( 1+\ln(T/\tau) \big) \nm{v_h}_{L^2(\Omega)}.
%\end{align*}
%%\end{small}
%%\begin{small}
%%\begin{align*}
%%\nm{W}_{{}^0H^{-\alpha/2}(0,T;\dot H^1(\Omega))} &
%%\leqslant C_\alpha \nm{v_h}_{L^2(\Omega)}, \\
%%\sum_{j=1}^J \nm{\jmp{W_j}}_{L^2(\Omega)} +
%%\nm{
%%Q_\tau \D_{T-}^{-\alpha}(-\Delta_h W)
%%}_{L^1(0,T;L^2(\Omega))} &
%%\leqslant C_\alpha \ln(T/\tau) \nm{v_h}_{L^2(\Omega)}.
%%\end{align*}
%\end{thm}

\begin{lem}
	\label{lem:731}
	If $ f \in {}_0H^{\alpha/2}(0,T;L^2(\Omega)) $, then
	\begin{equation}
		\label{eq:731}
		\begin{aligned}
			\nm{(U-P_\tau P_hu)_j}_{L^2(\Omega)}
			& \lesssim
			\ln(T/\tau) \nm{R_hu - P_\tau P_hu}_{L^\infty(0,T;L^2(\Omega))} \\
			& \qquad{} +
			\tau^{\alpha/2} \nm{(I-Q_\tau)u}_{L^2(0,T;\dot H^1(\Omega))}
		\end{aligned}
	\end{equation}
	for each $ 1 \leqslant j \leqslant J $.
\end{lem}
\begin{proof}
	Let $ \theta = U - P_\tau P_h u $ and set $ (P_\tau P_hu)_0 = 0 $.
	Define $ W \in W_{\tau,h} $ by that $ W_J^{+} = \theta_J $ and
	\[
		-\sum_{j=1}^J \dual{V_j, \jmp{W_j}}_\Omega +
		\dual{\nabla V, \nabla\D_{T-}^{-\alpha} W}_{\Omega \times (0,T)} = 0
		\quad \forall V \in W_{\tau,h}.
	\]
	A simple calculation then yields
	\begin{small}
	\begin{align*}
		& \nm{\theta_J}_{L^2(\Omega)}^2 =
		\dual{\theta_J, W_J^{+}}_\Omega =
		\sum_{j=0}^{J-1} \dual{\jmp{\theta_j}, W_j^{+}}_\Omega +
		\sum_{j=1}^J \dual{\theta_j, \jmp{W_j}}_\Omega \\
		={} &
		\sum_{j=0}^{J-1} \dual{\jmp{\theta_j}, W_j^{+}}_\Omega +
		\dual{\nabla\theta, \nabla\D_{T-}^{-\alpha} W}_{\Omega \times (0,T)} \\
		={} &
		\sum_{j=0}^{J-1} \dual{\jmp{\theta_j}, W_j^{+}}_\Omega +
		\dual{\nabla\D_{0+}^{-\alpha}\theta, \nabla W}_{\Omega \times (0,T)},
	\end{align*}
	\end{small}
	%Proceeding as in the proof of \cref{thm:stab} yields
	%\[
	%\nm{\theta_J}_{L^2(\Omega)}^2 =
	%\sum_{j=0}^{J-1} \dual{\jmp{\theta_j}, W_j^{+}}_\Omega +
	%\dual{\nabla\D_{0+}^{-\alpha}\theta, \nabla W}_{\Omega \times (0,T)},
	%\]
	and proceeding as in the proof of \cref{thm:conv_f_L2} yields
	\begin{small}
	\begin{align*}
		\sum_{j=0}^{J-1} \dual{\jmp{\theta_j}, W_j^{+}}_\Omega +
		\dual{\nabla\D_{0+}^{-\alpha} \theta, \nabla W}_{\Omega \times (0,T)} =
		\dual{\nabla\D_{0+}^{-\alpha} (u-P_\tau P_hu), \nabla W}_{\Omega \times (0,T)}.
	\end{align*}
	\end{small}
	Consequently,
	\begin{align}
		\nm{\theta_J}_{L^2(\Omega)}^2 &=
		\dual{
			\nabla(u-P_\tau P_hu), \nabla \D_{T-}^{-\alpha} W
		}_{\Omega \times (0,T)} \notag \\
		&= \dual{
			\nabla(R_hu - P_\tau P_hu),\ \nabla\D_{T-}^{-\alpha} W
		}_{\Omega \times (0,T)} \notag \\
		& = \dual{
			R_hu - P_\tau P_h u,\ \D_{T-}^{-\alpha}(-\Delta_h W
		}_{\Omega \times (0,T)} \notag \\
		&= \mathbb I_1 + \mathbb I_2, \label{eq:theta_J}
	\end{align}
	where
	\begin{align*}
		\mathbb I_1 & := \dual{
			R_hu - P_\tau P_hu,\ Q_\tau\D_{T-}^{-\alpha}(-\Delta_h W)
		}_{\Omega \times (0,T)}, \\
		\mathbb I_2 &:=
		\dual{
			R_hu-P_\tau P_hu,\ (I-Q_\tau)\D_{T-}^{-\alpha}(-\Delta_h W)
		}_{\Omega \times (0,T)}.
	\end{align*}

	Next, it is evident that
	\begin{equation}
		\label{eq:shit-21}
		\mathbb I_1 \leqslant
		\nm{R_hu - P_\tau P_hu}_{L^\infty(0,T;L^2(\Omega))}
		\nm{Q_\tau\D_{T-}^{-\alpha}(-\Delta_h W)}_{L^1(0,T;L^2(\Omega))}.
	\end{equation}
	By the definitions of $ Q_\tau $ and $ R_h $,
	\begin{small}
	\begin{align*}
		\mathbb I_2 &=
		\dual{
			R_hu, (I-Q_\tau) \D_{T-}^{-\alpha}(-\Delta_h W)
		}_{\Omega \times (0,T)} \\
		&=
		\dual{
			\nabla R_hu, \nabla(I-Q_\tau)\D_{T-}^{-\alpha}W
		}_{\Omega \times (0,T)} \\
		&=
		\dual{
			\nabla u, \nabla(I-Q_\tau)\D_{T-}^{-\alpha}W
		}_{\Omega \times (0,T)} \\
		&=
		\dual{
			\nabla(I-Q_\tau)u, \nabla(I-Q_\tau)\D_{T-}^{-\alpha}W
		}_{\Omega \times (0,T)} \\
		&\leqslant
		\nm{(I-Q_\tau)u}_{L^2(0,T;\dot H^1(\Omega))}
		\nm{(I-Q_\tau)\D_{T-}^{-\alpha}W}_{L^2(0,T;\dot H^1(\Omega))}.
	\end{align*}
	\end{small}
	In addition,
	\begin{align*}
		& \nm{(I-Q_\tau)\D_{T-}^{-\alpha}W}_{L^2(0,T;\dot H^1(\Omega))} \\
		\lesssim{} & \tau^{\alpha/2}
		\nm{\D_{T-}^{-\alpha} W}_{{}^0H^{\alpha/2}(0,T;\dot H^1(\Omega))}
		\quad\text{(by \cref{eq:Q_tau-sys})} \\
		\lesssim{} & \tau^{\alpha/2}
		\nm{W}_{{}^0H^{-\alpha/2}(0,T;\dot H^1(\Omega))}
		\quad\text{(by \cref{lem:regu}).}
	\end{align*}
	Consequently,
	\begin{equation}
		\label{eq:shit-22}
		\mathbb I_2 \lesssim
		\tau^{\alpha/2} \nm{(I-Q_\tau)u}_{L^2(0,T;\dot H^1(\Omega))}
		\nm{W}_{{}^0H^{-\alpha/2}(0,T;\dot H^1(\Omega))}.
	\end{equation}

	Finally, by the symmetric version of \cref{thm:discrete_regu} we have
	\begin{align*}
		\nm{W}_{{}^0H^{-\alpha/2}(0,T;\dot H^1(\Omega))}
		\leqslant C_\alpha \nm{\theta_J}_{L^2(\Omega)}, \\
		\nm{Q_\tau\D_{T-}^{-\alpha}(-\Delta_h W)}_{L^1(0,T;L^2(\Omega))}
		\leqslant C_\alpha\ln(T/\tau) \nm{\theta_J}_{L^2(\Omega)},
	\end{align*}
	and hence combining \cref{eq:theta_J,eq:shit-21,eq:shit-22} yields that
	\cref{eq:731} holds for $ j = J $. Since the case $ 1 \leqslant j < J $ can be
	proved analogously, this completes the proof.
\end{proof}

%\begin{lem}
%\label{lem:appro_err}
%If $ f \in {}_0H^{\alpha+1/2}(0,T;L^2(\Omega)) $, then
%\begin{equation}
%\label{eq:appro_err}
%\begin{aligned}
%& \nm{(I-R_h)u}_{L^\infty(0,T;L^2(\Omega))} +
%\nm{(I-P_h)u}_{L^\infty(0,T;L^2(\Omega))} \\
%\lesssim{} &
%\sqrt{\ln(1/h)}\, h^2
%\nm{f}_{ {}_0H^{\alpha+1/2}(0,T;L^2(\Omega)) }.
%\end{aligned}
%\end{equation}
%\end{lem}
%\begin{proof}
%By \cref{thm:regu-pde},
%\begin{align*}
%& \nm{(I-R_h)u}_{L^\infty(0,T;L^2(\Omega))} +
%\nm{(I-P_h)u}_{L^\infty(0,T;L^2(\Omega))} \\
%\lesssim{} &
%h^{2(1-\epsilon)} \nm{u}_{C([0,T];\dot H^{2(1-\epsilon)}(\Omega))} \\
%\lesssim{} &
%\frac{h^{2(1-\epsilon)}}{\sqrt\epsilon}
%\nm{f}_{{}_0H^{\alpha+1/2}(0,T;L^2(\Omega))}
%\end{align*}
%for all $ 0 < \epsilon < 1/2 $. By the assumption $ h < e^{-2(1+\alpha)} $ (cf.~the
%first paragraph of \cref{sec:main}), letting $ \epsilon := (\ln(1/h))^{-1} $ then
%yields \cref{eq:appro_err} and thus concludes the proof.
%\end{proof}

Finally, we conclude the proof of \cref{thm:conv_f_higher} as follows. By
\cref{lem:731}, a straightforward computation yields
\begin{small}
\begin{align}
	& \nm{u-U}_{L^\infty(0,T;L^2(\Omega))} \notag \\
	\lesssim{} &
	\tau^{\alpha/2} \nm{(I-Q_\tau)u}_{L^2(0,T;\dot H^1(\Omega))} +
	\ln(T/\tau) \Big(
		\nm{(I-R_h)u}_{L^\infty(0,T;L^2(\Omega))} \notag \\
		& \quad{} +
		\nm{(I-P_h)u}_{L^\infty(0,T;L^2(\Omega))} +
		\nm{(I-P_\tau)u}_{L^\infty(0,T;L^2(\Omega))}
	\Big). \label{eq:shit-1}
\end{align}
\end{small}
By \cref{thm:regu-pde} we have
\begin{align*}
	& \nm{(I-R_h)u}_{L^\infty(0,T;L^2(\Omega))} +
	\nm{(I-P_h)u}_{L^\infty(0,T;L^2(\Omega))} \\
	\lesssim{} &
	h^{2(1-\epsilon)} \nm{u}_{C([0,T];\dot H^{2(1-\epsilon)}(\Omega))} \\
	\lesssim{} &
	\frac{h^{2(1-\epsilon)}}{\sqrt\epsilon}
	\nm{f}_{{}_0H^{\alpha+1/2}(0,T;L^2(\Omega))}
\end{align*}
for all $ 0 < \epsilon < 1/2 $, so that, by the assumption $ h < e^{-2(1+\alpha)} $
(cf.~the first paragraph of \cref{sec:main}), letting $ \epsilon := (\ln(1/h))^{-1} $
yields
\begin{equation}
	\label{eq:shit-2}
	\begin{aligned}
		& \nm{(I-R_h)u}_{L^\infty(0,T;L^2(\Omega))} +
		\nm{(I-P_h)u}_{L^\infty(0,T;L^2(\Omega))} \\
		\lesssim{} &
		\sqrt{\ln(1/h)}\, h^2
		\nm{f}_{ {}_0H^{\alpha+1/2}(0,T;L^2(\Omega)) }.
	\end{aligned}
\end{equation}
In addition, by \cref{thm:regu-pde,lem:interp}, it is standard that
\begin{align}
	& \nm{(I-Q_\tau)u}_{L^2(0,T;\dot H^1(\Omega))} +
	\nm{(I-P_\tau)u}_{L^\infty(0,T;L^2(\Omega))} \notag \\
	\lesssim{} &
	\tau \nm{f}_{{}_0H^{\alpha+1/2}(0,T;L^2(\Omega))}.
	\label{eq:shit-3}
\end{align}
Combining \cref{eq:shit-1,eq:shit-2,eq:shit-3} proves \cref{eq:conv_f_higher} and thus
concludes the proof of \cref{thm:conv_f_higher}.

\section{Numerical experiments}
\label{sec:numer}
This section performs four numerical experiments in one dimensional space to verify
\cref{thm:conv-u0,thm:f-const,thm:conv_f_L2,thm:conv_f_higher}, respectively.
Throughout this section, $ \Omega = (0,1) $, $ T = 1 $, the spatial and temporal
grids are both uniform, and $ U^{m,n} $ is the numerical solution with $ h = 2^{-m} $
and $ \tau = 2^{-n} $. Additionally, $ \nm{\cdot}_{L^\infty(0,T;L^2(\Omega))} $ is
abbreviated to $ \nm{\cdot} $ for convenience, and, for any $ \beta > 0 $,
\[
	\nm{v}_{\beta,n} := \max_{1 \leqslant j \leqslant 2^n}
	(j/2^n)^\beta \nm{v((j/2^n)-)}_{L^2(\Omega)},
\]
where $ v((j/2^n)-) $ means the left limit of $ v $ at $ j/2^n $.

\medskip\noindent {\bf Experiment 1.} This experiment verifies \cref{thm:conv-u0} in
the setting
\[
	u_0(x) = x^{-0.49}, \quad x \in \Omega,
\]
which is slightly smoother than $ L^2(\Omega) $. \cref{tab:ex3-time} validates the
theoretical prediction that the convergence behavior of $ U $ is close to $ \mathcal
O(\tau) $ when $ h $ is fixed and sufficiently small. \cref{tab:ex3-space} confirms
the theoretical prediction that the convergence behavior of $ U $ is close to $
\mathcal O(h^2) $ when $ \tau $ is fixed and sufficiently small.
\begin{table}[H]
	\caption{Convergence behavior with respect to $ \tau $.}
	\label{tab:ex3-time}
	\footnotesize
	\setlength{\tabcolsep}{3pt}
	\begin{tabular}{ccccccc}
		\toprule &
		\multicolumn{2}{c}{$\alpha=0.2$} &
		\multicolumn{2}{c}{$\alpha=0.4$} &
		\multicolumn{2}{c}{$\alpha=0.8$} \\
		\cmidrule(r){2-3}
		\cmidrule(r){4-5}
		\cmidrule(r){6-7}
		$n$ & $ \|U^{11,n}\!-\!U^{11,16}\|_{1,n} $ & Order
		& $ \|U^{11,n}\!-\!U^{11,16}\|_{1,n} $ & Order
		& $ \|U^{11,n}\!-\!U^{11,16}\|_{1,n} $ & Order \\
		$6$ & 9.07e-3 & --   & 1.44e-2 & --   & 7.05e-2 & --   \\
		$7$ & 4.58e-3 & 0.98 & 7.27e-3 & 0.98 & 3.93e-2 & 0.84 \\
		$8$ & 2.30e-3 & 0.99 & 3.66e-3 & 0.99 & 2.10e-2 & 0.91 \\
		$9$ & 1.15e-3 & 1.00 & 1.83e-3 & 1.00 & 1.09e-2 & 0.95 \\
		\bottomrule
	\end{tabular}
\end{table}

\begin{table}[H]
	\caption{Convergence behavior with respect to $ h $.}
	\label{tab:ex3-space}
	\footnotesize
	\setlength{\tabcolsep}{2pt}
	\begin{tabular}{ccccccc}
		\toprule &
		\multicolumn{2}{c}{$\alpha=0.2$} &
		\multicolumn{2}{c}{$\alpha=0.4$} &
		\multicolumn{2}{c}{$\alpha=0.8$} \\
		\cmidrule(r){2-3}
		\cmidrule(r){4-5}
		\cmidrule(r){6-7}
		$m$ & $ \|U^{m,16}\!-\!U^{11,16}\|_{1.2,16} $ & Order
		& $ \|U^{m,16}\!-\!U^{11,16}\|_{1.4,16} $ & Order
		& $ \|U^{m,16}\!-\!U^{11,16}\|_{1.8,16} $ & Order \\
		$3$ & 1.43e-3 & --   & 4.51e-3 & --   & 7.06e-2 & --   \\
		$4$ & 3.62e-4 & 1.98 & 1.13e-3 & 1.99 & 2.37e-2 & 1.57 \\
		$5$ & 9.13e-5 & 1.99 & 2.83e-4 & 2.00 & 6.76e-3 & 1.81 \\
		$6$ & 2.30e-5 & 1.99 & 7.09e-5 & 2.00 & 1.74e-3 & 1.96 \\
		\bottomrule
	\end{tabular}
\end{table}

\medskip\noindent {\bf Experiment 2.} This experiment verifies \cref{thm:f-const} in
the setting
\[
	v(x) = x^{-0.49}, \quad x \in \Omega.
\]
\cref{tab:ex4-time} confirms the theoretical prediction that the convergence behavior
of $ U $ is close to $ \mathcal O(\tau) $ when $ h $ is fixed and sufficiently small.
\cref{tab:ex4-space} confirms the theoretical prediction that the accuracy of $ U(T-)
$ (the left limit of $ U $ at $ T $) in the norm $ \nm{\cdot}_{L^2(\Omega)} $ is
close to $ \mathcal O(h^2) $ when $ \tau $ is fixed and sufficiently small.

\begin{table}[H]
	\caption{Convergence behavior with respect to $ \tau $.}
	\label{tab:ex4-time}
	\footnotesize
	\begin{tabular}{ccccccc}
		\toprule &
		\multicolumn{2}{c}{$\alpha=0.2$} &
		\multicolumn{2}{c}{$\alpha=0.4$} &
		\multicolumn{2}{c}{$\alpha=0.8$} \\
		\cmidrule(r){2-3}
		\cmidrule(r){4-5}
		\cmidrule(r){6-7}
		$n$ & $ \|U^{11,n}\!-\!U^{11,16}\| $ & Order
		& $ \|U^{11,n}\!-\!U^{11,16}\| $ & Order
		& $ \|U^{11,n}\!-\!U^{11,16}\| $ & Order \\
		$6$ & 4.53e-3 & --   & 5.45e-3 & --   & 1.01e-2 & --   \\
		$7$ & 2.31e-3 & 0.97 & 2.77e-3 & 0.97 & 5.36e-3 & 0.91 \\
		$8$ & 1.17e-3 & 0.99 & 1.40e-3 & 0.99 & 2.78e-3 & 0.95 \\
		$9$ & 5.85e-4 & 1.00 & 7.00e-4 & 1.00 & 1.41e-3 & 0.97 \\
		\bottomrule
	\end{tabular}
\end{table}

\begin{table}[H]
	\caption{Convergence behavior with respect to $ h $.}
	\label{tab:ex4-space}
	\footnotesize\setlength{\tabcolsep}{4pt}
	\begin{tabular}{ccccc}
		\toprule &
		\multicolumn{2}{c}{$\alpha=0.2$} &
		\multicolumn{2}{c}{$\alpha=0.8$} \\
		\cmidrule(r){2-3}
		\cmidrule(r){4-5}
		$m$ & $ \|(U^{m,16}-U^{11,16})(T-)\|_{L^2(\Omega)} $ & Order
		& $ \|(U^{m,16}\!-\!U^{11,16})(T-)\|_{L^2(\Omega)} $ & Order \\
		$3$ & 2.71e-3 & --    & 7.76e-3 & --   \\
		$4$ & 7.21e-4 & 1.91  & 2.04e-3 & 1.93 \\
		$5$ & 1.90e-4 & 1.92  & 5.09e-4 & 2.00 \\
		$6$ & 4.97e-5 & 1.93  & 1.27e-4 & 2.00 \\
		\bottomrule
	\end{tabular}
\end{table}
%\begin{table}[H]
%\caption{Convergence behavior with respect to $ h $}
%\label{tab:ex4-space}
%\footnotesize\setlength{\tabcolsep}{1pt}
%\begin{tabular}{ccccccc}
%\toprule &
%\multicolumn{2}{c}{$\alpha=0.2$} &
%\multicolumn{2}{c}{$\alpha=0.4$} &
%\multicolumn{2}{c}{$\alpha=0.8$} \\
%\cmidrule(r){2-3}
%\cmidrule(r){4-5}
%\cmidrule(r){6-7}
%$m$ & $ \|U^{m,16}(T-)-U^{11,16}(T-)\| $ & Order
%& $ \|U^{m,16}(T-)\!-\!U^{11,16}(T-)\| $ & Order
%& $ \|U^{m,16}(T-)\!-\!U^{11,16}(T-)\| $ & Order \\
%$3$ & 2.71e-3 & --   & 2.33e-3 & --   & 7.76e-3 & --   \\
%$4$ & 7.21e-4 & 1.91 & 6.15e-4 & 1.92 & 2.04e-3 & 1.93 \\
%$5$ & 1.90e-4 & 1.92 & 1.61e-4 & 1.93 & 5.09e-4 & 2.00 \\
%$6$ & 4.97e-5 & 1.93 & 4.19e-5 & 1.94 & 1.27e-4 & 2.00 \\
%\bottomrule
%\end{tabular}
%\end{table}

\medskip\noindent {\bf Experiment 3.} This experiment verifies \cref{thm:conv_f_L2} in
the setting
\[
	f(x,t) = x^{\alpha/(\alpha+1)-0.49} t^{-0.49},
	\quad (x,t) \in \Omega \times (0,T),
\]
which has slightly higher regularity than $ L^2(0,T;\dot
H^{\alpha/(\alpha+1)}(\Omega)) $. \cref{thm:conv_f_L2} predicts that the convergence
behavior of $ U $ is close to $ \mathcal O(h) $ when $ \tau $ is fixed and
sufficiently small, and this is in good agreement with the numerical results in
\cref{tab:ex1-space}. Moreover, \cref{thm:conv_f_L2} predicts that the convergence
behavior of $ U $ is close to $ \mathcal O(\tau^{1/2}) $ when $ h $ is fixed and
sufficiently small, which agrees well with the numerical results in
\cref{tab:ex1-time}.

\renewcommand{\arraystretch}{1.5}
\setlength{\tabcolsep}{5pt}
\begin{table}[H]
	\caption{Convergence behavior with respect to $ h $.}
	\label{tab:ex1-space}
	\footnotesize
	\begin{tabular}{ccccccc}
		\toprule &
		\multicolumn{2}{c}{$\alpha=0.2$} &
		\multicolumn{2}{c}{$\alpha=0.4$} &
		\multicolumn{2}{c}{$\alpha=0.8$} \\
		\cmidrule(r){2-3}
		\cmidrule(r){4-5}
		\cmidrule(r){6-7}
		$m$ & $ \|U^{m,16}\!-\!U^{11,16}\| $ & Order
		& $ \|U^{m,16}\!-\!U^{11,16}\| $ & Order
		& $ \|U^{m,16}\!-\!U^{11,16}\| $ & Order \\
		$3$ & 3.53e-2 & --   & 3.84 e-2 & --   & 4.95e-2 & --   \\
		$4$ & 1.70e-2 & 1.05 & 1.85e-2  & 1.06 & 2.41e-2 & 1.04 \\
		$5$ & 8.22e-3 & 1.05 & 8.89e-3  & 1.05 & 1.17e-2 & 1.04 \\
		$6$ & 3.95e-3 & 1.06 & 4.29e-3  & 1.05 & 5.69e-3 & 1.04 \\
		\bottomrule
	\end{tabular}
\end{table}

\begin{table}[H]
	\caption{Convergence behavior with respect to $ \tau $.}
	\label{tab:ex1-time}
	\footnotesize
	\begin{tabular}{ccccccc}
		\toprule &
		\multicolumn{2}{c}{$\alpha=0.2$} &
		\multicolumn{2}{c}{$\alpha=0.4$} &
		\multicolumn{2}{c}{$\alpha=0.8$} \\
		\cmidrule(r){2-3}
		\cmidrule(r){4-5}
		\cmidrule(r){6-7}
		$n$ & $ \|U^{11,n}\!-\!U^{11,16}\| $ & Order
		& $ \|U^{11,n}\!-\!U^{11,16}\| $ & Order
		& $ \|U^{11,n}\!-\!U^{11,16}\| $ & Order \\
		$6$ & 2.75e-1 & --   & 2.58e-1 & --   & 2.32e-1 & --   \\
		$7$ & 2.04e-1 & 0.43 & 1.86e-1 & 0.47 & 1.63e-1 & 0.51 \\
		$8$ & 1.48e-1 & 0.47 & 1.32e-1 & 0.50 & 1.13e-1 & 0.53 \\
		$9$ & 1.05e-1 & 0.50 & 9.21e-2 & 0.52 & 7.78e-2 & 0.54 \\
		\bottomrule
	\end{tabular}
\end{table}

\medskip\noindent {\bf Experiment 4.} This experiment verifies
\cref{thm:conv_f_higher} in the setting
\[
	f(x,t) = x^{-0.49} t^{\alpha+0.01},
	\quad (x,t) \in \Omega \times (0,T),
\]
which is slightly smoother than $ {}_0H^{\alpha+1/2}(0,T;L^2(\Omega)) $.
\cref{tab:ex2-space} confirms the theoretical prediction that the convergence
behavior of $ U $ is close to $ \mathcal O(h^2) $ when $ \tau $ is fixed and
sufficiently small, and \cref{tab:ex2-time} confirms the theoretical prediction that
the convergence behavior of $ U $ is close to $ \mathcal O(\tau) $ when $ h $ is
fixed and sufficiently small.

\begin{table}[H]
	\caption{Convergence behavior with respect to $ h $.}
	\label{tab:ex2-space}
	\footnotesize
	\begin{tabular}{ccccccc}
		\toprule &
		\multicolumn{2}{c}{$\alpha=0.2$} &
		\multicolumn{2}{c}{$\alpha=0.4$} &
		\multicolumn{2}{c}{$\alpha=0.8$} \\
		\cmidrule(r){2-3}
		\cmidrule(r){4-5}
		\cmidrule(r){6-7}
		$m$ & $ \|U^{m,16}\!-\!U^{11,16}\| $ & Order
		& $ \|U^{m,16}\!-\!U^{11,16}\| $ & Order
		& $ \|U^{m,16}\!-\!U^{11,16}\| $ & Order \\
		$3$ & 2.90e-3 & --   & 3.14e-3 & --   & 4.46e-3 & --   \\
		$4$ & 7.70e-4 & 1.91 & 8.23e-4 & 1.93 & 1.15e-3 & 1.95 \\
		$5$ & 2.03e-4 & 1.92 & 2.15e-4 & 1.94 & 2.97e-4 & 1.96 \\
		$6$ & 5.36e-5 & 1.92 & 5.59e-5 & 1.94 & 7.75e-5 & 1.94 \\
		\bottomrule
	\end{tabular}
\end{table}

\begin{table}[H]
	\caption{Convergence behavior with respect to $ \tau $.}
	\label{tab:ex2-time}
	\footnotesize
	\begin{tabular}{ccccccc}
		\toprule &
		\multicolumn{2}{c}{$\alpha=0.2$} &
		\multicolumn{2}{c}{$\alpha=0.4$} &
		\multicolumn{2}{c}{$\alpha=0.8$} \\
		\cmidrule(r){2-3}
		\cmidrule(r){4-5}
		\cmidrule(r){6-7}
		$n$ & $ \|U^{11,n}\!-\!U^{11,16}\| $ & Order
		& $ \|U^{11,n}\!-\!U^{11,16}\| $ & Order
		& $ \|U^{11,n}\!-\!U^{11,16}\| $ & Order \\
		$6$ & 8.98e-3 & --   & 6.15e-3 & --   & 5.24e-3 & --   \\
		$7$ & 4.54e-3 & 0.98 & 3.10e-3 & 0.99 & 2.64e-3 & 0.99 \\
		$8$ & 2.28e-3 & 0.99 & 1.56e-3 & 0.99 & 1.33e-3 & 1.00 \\
		$9$ & 1.14e-3 & 1.00 & 7.78e-4 & 1.00 & 6.62e-4 & 1.00 \\
		\bottomrule
	\end{tabular}
\end{table}

\section{Conclusion}
\label{sec:conclusion}
A time-stepping discontinuous Galerkin method is analyzed in this paper. Nearly
optimal error estimate with respect to the regularity of the solution is derived with
nonsmooth source term, nearly optimal error estimate is derived when the source term
satisfies some regularity assumption, and error estimate with nonsmooth initial vaue
is derived by the Laplace transform technique. In addition, the effect of the
nonvanishing $ f(0) $ on the accuracy of the numerical solution is also investigated.
Finally, numerical results are provided to verify the theoretical results.

%However, higher-order version of the algorithm analyzed in this paper has not been
%considered, and this appears to be rather complicated. How to generalize our
%theoretical results to higher-order version is our ongoing work.

\end{document}